\documentclass[12pt]{article}
\usepackage{amsmath,amssymb,amsthm,mathtools,enumitem, hyperref}
\usepackage{authblk}
\usepackage{color}
\usepackage{geometry}
\usepackage[mathscr]{eucal}
\usepackage[T1]{fontenc}
\usepackage{lmodern}
\newtheorem{Theorem}{Theorem}[section]

\newtheorem{Lemma}{Lemma}[section]
\newtheorem{Proposition}{Proposition}[section]
\newtheorem{Remark}{Remark}[section]
\newtheorem{Definition}{Definition}[section]

\begin{document}

\def\eR{\mathbf{R}}
\def\Rd{{\eR}^d}
\def\Rdd{{\eR}^{d\times d}}
\def\Rdsym{{\eR}^{d\times d}_{sym}}
\def\eN{\mathbb{N}}
\def\eZ{\mathbf{Z}}
\def\IdM{\mathbb{I}_d}
\def\CrlS{\mathcal{S}}
\def\tder{\partial_t}
\def\SymG{\mathbb{D}}
\def\SymGDev{\mathbb{D}^{D}}
\def\dd{\mbox{d}}
\newcommand{\essinf}{\operatorname{ess\,inf}}
\newcommand{\esssup}{\operatorname{ess\,sup}}
\newcommand{\supp}{\operatorname{supp}}
\newcommand{\spn}{\operatorname{span}}
\newcommand\dx{\; \mathrm{d}x}
\newcommand\dy{\; \mathrm{d}y}
\newcommand\dz{\; \mathrm{d}z}
\newcommand\dt{\; \mathrm{d}t}
\newcommand\ds{\; \mathrm{d}s}
\newcommand\diff{\mathrm{d}}
\newcommand\dvr{\mathop{\mathrm{div}}\nolimits}
\newcommand\pat{\partial_t}
\newcommand\sym{\mathrm{sym}}
\newcommand\diam{\mathrm{diam}}
\newcommand\tr{\mathop{\mathrm{tr}}\nolimits}
\newcommand\lspan{\mathop{\mathrm{span}}\nolimits}

\title{Local-in-time existence of strong solutions to a class of the compressible non-Newtonian Navier-Stokes equations}
\author{Martin Kalousek\footnote{kalousek@math.cas.cz, Institute of Mathematics, Czech Academy of Sciences, \v Zitn\'a 25, 115 67 Praha 1, Czech Republic}, V\'aclav M\'acha\footnote{macha@math.cas.cz, Institute of Mathematics, Czech Academy of Sciences, \v Zitn\'a 25, 115 67 Praha 1, Czech Republic}, \v{S}\'arka Ne\v{c}asov\'a \footnote{matus@math.cas.cz, Institute of Mathematics, Czech Academy of Sciences, \v Zitn\'a 25, 115 67 Praha 1, Czech Republic}}
\maketitle
\begin{abstract} 
The aim of this article is to show a local-in-time existence of a strong solution to the generalized compressible Navier-Stokes equations for arbitrarily large initial data. The goal is reached by $L^p$-theory for linearized equations which are obtained with help of the Weis multiplier theorem and can be seen as a generalization of the work of Shibata and Enomoto \cite{EnSh} (devoted to compressible fluids) to compressible non-Newtonian fluids. 
\end{abstract}
\section{Introduction}
In this paper, we analyze a boundary value problem for the Navier-Stokes equations describing the flow of a compressible non-Newtonian fluid that reads
\begin{equation}\label{SystemClass}
	\begin{alignedat}{2}
	\tder(\varrho u)+\dvr(\varrho u\otimes u)+\nabla \pi(\varrho)=&\dvr\CrlS &&\text{ in }Q_T,\\
	\tder\varrho+\dvr(\varrho u)=&0 &&\text{ in }Q_T,\\
	u(0,\cdot)=u_0,\varrho(0,\cdot)=&\varrho_0&&\text{ in }\Omega. 
	\end{alignedat}
\end{equation}
Here $\Omega\subset\mathbb{R}^d,d\geq 2$ is the domain occupied by the fluid, $T>0$ is the time of evolution and $Q_T=(0,T)\times\Omega$. We denote by $u(t,x)$ the fluid velocity, by $\varrho(t,x)$ the fluid density and by $\CrlS$ the viscous stress. We assume the periodic boundary condition, i.e., $\Omega$ is a torus
\begin{equation*}
\Omega = \left([-1,1]|_{\{-1,1\}}\right)^d.
\end{equation*}

We restrict ourselves to the constitutive relation
\begin{equation}\label{SDef}
\CrlS=2\mu(|\SymGDev u|^2)\SymGDev u+\eta(\dvr u)\dvr u\IdM,
\end{equation}
where $\mu$ and $\eta$ are viscosity coefficients whose properties will be specified later.

\subsection{Discussion and main result}
The system with $\mu \equiv \mu_0\in \mathbb R$ and $\eta\equiv \eta_0\in \mathbb R$ has been extensively studied. The mathematical analysis of compressible viscous fluids goes back to 1950s. The first result concerning uniqueness was given by Graffi \cite{Graffi} and Serrin \cite{Serrin}. A local in time existence in H\" older continuous spaces was proven by Nash \cite {Nash}, Itaya \cite{Itaya1,Itaya2} and Vol'pert and Hudjaev \cite {Volpert}.  In Sobolev-Slobodetskii space the local existence was shown by Solonnikov \cite{Sol}. The local-in-time existence and global-in-time existence for small data in Hilbert space traces back to Valli \cite{Valli}. The optimal regularity of local in time solutions was obtained by Charve and Danchin \cite{ChaDa}. The global solution with the initial data close to a rest state was shown by Matsumura and Nishida \cite{MatNi1,MatNi2} in the whole space and the exterior domain. The case of bounded domain was treated in work of Valli and Zaj\k{a}czkowski \cite{VaZa}.
 
The global-in-time existence of a weak solution is much more recent and goes back to the half of 1990s. The existence is known due to nowadays standard Feireisl - Lions theory -- we refer to \cite{FeNoPe} and to \cite{Lions}. 

However, there is an unsatisfactory number of articles for the system with a general $\mu$ and $\eta$. In particular, Mamontov \cite{Mamontov1}, \cite{Mamontov2} considered an exponentially growing viscosity and isothermal pressure. For such case the global existence of a weak solution was shown. Recently, Abbatiello, Feireisl and Novotn\'y in \cite{AbFeNo} provide a proof of the existence of so-called dissipative solution -- a notion of solution which is 'weaker' than a weak solution. The existence of weak solution is still unproven. The main obstacle seems to be the lack of compactness of velocity $u$ which is caused by an insufficient control of $\pat u$ in (and near) vacuum regions. 

Let us mention that such type of problem was already studied in the incompressible case by Ladyzhenskaya \cite{Lad} and then extensively studied by group of J. Ne\v cas \cite{Bel1, Bel2, MNR}. Further, basic properties of weak or measure-valued solution of incompressible non-Newtonian fluids are described in \cite{MNRR,BeBl}. Concerning the compressible case the first attempt to solve this problem can be found in work of Ne\v casov\' a and Novotn\' y \cite{MNN} where the existence of measure-valued solution was shown. 

The concept of the maximal regularity is classical and the main achievements for abstract theory go back to  the  work of Ladyzhenskaya, Uraltzeva and Solonnikov \cite{LUS}, Da-Prato and Grisvard \cite{DG}, Amann \cite {Amann2} and Pr{\" u}ss \cite{P}.  The notion of  $\cal R$--sectoriality was introduced  by Clement and Pr\"uss see \cite{CP}. The fundamental result by Weis \cite{W} about equivalency   of description for the maximal regularity in terms of vector-valued Fourier multipliers and ${\cal R}$--sectoriality was a breakthrough for the applications. 

The Weis multiplier theorem allows to deduce the $L^p$ maximal regularity to problems connected with the fluid dynamics. We refer to \cite{DeHiPr} and \cite{KuWe} for the comprehensible decription of the method.
The maximal $L^p$ regularity result for incompressible fluids can be found in work by Shibata and Shimizu \cite{SHS}, for the  non-Newtonian incompressible situation it was established in work by Bothe and Pr\"uss \cite{BP}. Maximal regularity for the compressible case was shown by Enomoto and Shibata \cite{EnSh}.
 
The goal of this paper is to use the Weis multiplier theorem  to show the short-time existence of a strong solution in $L^p$ setting for the non-Newtonian compressible fluid.


Let us introduce our main result.
\begin{Theorem}\label{Thm:Main}
Let $\mu \in C^3([0,\infty))$ and $\eta \in C^2(\mathbb{R})$ satisfy $\mu(s) + 2\mu'(s) s >0$ for all $s\geq 0$ and $\eta(r) + \eta'(r)r >0$ for all $r\in\mathbb{R}$. Let, moreover, $\pi \in C^2([0,\infty))$, $q>d$ and $p\in \left(\frac{2q}{q-d},\infty\right)$ be given. Then for every $u_0\in W^{2,q}(\Omega)$ and $\varrho_0\in W^{1,q}(\Omega)$, $\frac 1{\varrho_0}\in L^\infty(\Omega)$ there is $T>0$ such that there exists
\begin{equation*}
(\varrho,u)\in L^p(0,T;W^{1,q}(\Omega))\times L^p(0,T;W^{2,q}(\Omega))
\end{equation*}
with
\begin{equation*}
(\pat \varrho,\pat u)\in L^p(0,T;W^{1,q}(\Omega))\times L^p(0,T;L^q(\Omega))  
\end{equation*}
which satisfies \eqref{SystemClass}.\label{thm.main}
\end{Theorem}

The proof of the main result relies on several steps. The first step is the use of Lagrange coordinates to get rid of the convective term.

Secondly, we linearize the resulting system 
and we show $L^p-L^q$ regularity property of a solution. This is discussed in Section \ref{Lpreg} and it is obtained by means of the Weis multiplier theorem.

The sufficient regularity of the velocity $u$ allows to deduce appropriate bounds 
and one may use the Banach fix-point argument to deduce the existence of a strong solution assuming the time $T>0$ is small enough. This is described in Section \ref{Bfp}. 

The remaining part of this introductory section is devoted to necessary preliminary statements.

\subsection{Notation}
We denote by $\IdM$ the $d\times d$ identity matrix. For the $d\times d$--matrix valued mapping $G$ $\sym G$ denotes the symmetric part of $G$, i.e., $\sym G=\frac{1}{2}\left(G+G^\top\right)$ and $G^{D}$ is the traceless part of $G$, i.e., $G^{D}=G-\frac{1}{d}\tr G\IdM$. Let $u$ be a mapping with values in $\mathbb{R}^d$ then $\SymG u=\sym(\nabla u)$ is the symmetric part of the gradient and $\SymGDev u$ stands for the traceless part of $\SymG u$.
For purposes of this paper we denote by $\mathcal{V}^{p,q}(Q_T)$ the following function space
\begin{equation*}
\mathcal{V}^{p,q}(Q_T)=L^p(0,T;W^{2,q}(\Omega))\cap W^{1,p}(0,T;L^q(\Omega))
\end{equation*}
with the norm $\|\cdot\|_{V^{p,q}(Q_T)}=\|\cdot\|_{L^p(0,T;W^{2,q}(\Omega))}+\|\cdot\|_{W^{1,p}(0,T;L^q(\Omega))}$.
Starting with definition \eqref{SDef}, using the fact that $\tr\SymGDev u=0$ and the symmetry of $\SymGDev u$ we have 
\begin{equation*}
\begin{split}
(\dvr \CrlS)_j=&2\mu(|\SymGDev u|^2)\sum_{l=1}^d\partial_l \left((\SymGDev u)_{jl} -\frac{1}{d}\delta_{jl}\dvr u\right)\\
&+4\mu'(|\SymGDev u|^2)\sum_{k,l,m=1}^d(\SymGDev u)_{jk}(\SymGDev u)_{lm}\partial_k\left((\SymG u)_{lm}-\frac{1}{d}\delta_{lm}\dvr u\right)\\
&+\sum_{l=1}^d\delta_{jl}\left(\eta(\dvr u)+\eta '(\dvr u)\dvr u\right)\partial_l\dvr u\\
=&\mu(|\SymGDev u|^2)\sum_{l=1}^d(\partial^2_l u_j+\partial_j\partial_l u_l)+2\mu'(|\SymGDev u|^2)\sum_{k,l,m=1}^d(\SymGDev u)_{jk}(\SymGDev u)_{lm}\left(\partial_k\partial_m u_l+\partial_k\partial_lu_m\right)\\
&-\frac{2}{d}\mu(|\SymGDev u|^2)\partial_j\dvr u+\left(\eta(\dvr u)+\eta'(\dvr u)\dvr u\right)\partial_j \dvr u\\
=&\mu(|\SymGDev u|^2)\sum_{l=1}^d(\partial^2_l u_j+\partial_j\partial_l u_l)+4\mu'(|\SymGDev u|^2)\sum_{k,l,m=1}^d(\SymGDev u)_{jl}(\SymGDev u)_{km}\partial_l\partial_m u_k\\
&-\frac{2}{d}\mu(|\SymGDev u|^2)\partial_j\dvr u+\left(\eta(\dvr u)+\eta'(\dvr u)\dvr u\right)\partial_j \dvr u.
\end{split}
\end{equation*}
Hence we deduce that
\begin{equation}\label{eq:DivS}
(\dvr \CrlS)_j=\sum_{k,l,m=1}^d a^{lm}_{jk}(\mathbb Du)\partial_m\partial_l u_k,
\end{equation}
where
\begin{equation}\label{eq:ACoef}
\begin{split}
a^{lm}_{jk}(\mathbb Du)=&\mu(|\SymGDev u|^2)\left(\delta_{jk}\delta_{lm}+\delta_{jm}\delta_{kl}\right)+4\mu'(|\SymGDev u|^2)(\SymGDev u)_{jl}(\SymGDev u)_{km}\\
&+\left(\eta(\dvr u)+\eta'(\dvr u)\dvr u-\frac{2}{d}\mu(|\SymGDev u|^2)\right)\delta_{km}\delta_{jl}.
\end{split}
\end{equation}
We note that $a_{jk}^{lm}$ possesses the following symmetries
\begin{equation}\label{eq:a.symmetry}
a_{jk}^{lm}=a_{kj}^{lm}=a^{jk}_{lm}=a^{jm}_{lk}=a^{lk}_{jm}\text{ for all }j,k,l,m=1,\dots,d.
\end{equation}
We define for a given $u\in C^1(\overline{\Omega})^d$
the quasilinear differential operator $\mathcal{A}$ as
\begin{equation}\label{eq:def.A}
\mathcal{A}(\mathbb D u)(\mathbb Dv)=\sum_{l,m=1}^da^{lm}_{jk}(\mathbb Du)\partial_l\partial_m v,
\end{equation}
whose ellipticity is ensured by certain identities satisfied by $\mu,\eta$ and their derivatives as we now show.
Considering $\xi\in\mathbb{R}^{d\times d}_{sym}$ we get 
\begin{equation*}
\begin{split}
a^{lm}_{jk}\xi_{jl}\xi_{km}=& \mu(|\SymGDev u|^2)(\xi_{jl}^2+\xi^2_{km})+4\mu'(|\SymGDev u|^2)(\SymGDev u)_{jl}\xi_{jl}(\SymGDev u)_{km}\xi_{km}\\
&+\left(\eta(\dvr u)+\eta'(\dvr u)\dvr u-\frac{2}{d}\mu(|\SymGDev u|^2)\right)\xi_{jj}\xi_{ll}.
\end{split}
\end{equation*}
Hence we obtain
\begin{equation*}
\begin{split}
\sum_{j,k,l,m=1}^da^{lm}_{jk}\xi_{jl}\xi_{km}=&2\mu(|\SymGDev u|^2)|\xi|^2+4\mu'(|\SymGDev u|^2)|\SymGDev u\cdot\xi|^2\\
&+\left(\eta(\dvr u)+\eta'(\dvr u)\dvr u-\frac{2}{d}\mu(|\SymGDev u|^2)\right)(\tr \xi)^2.
\end{split}
\end{equation*}
Applying the decomposition $\xi=\xi^D+\frac{tr\xi}{d}\IdM$, in particular the fact that $\xi^D\cdot\IdM=0$, we arrive at 
\begin{equation}\label{BilFormDet}
\sum_{j,k,l,m=1}^da^{lm}_{jk}\xi_{jl}\xi_{km}=2\mu(|\SymGDev u|^2)|\xi^D|^2+4\mu'(|\SymGDev u|^2)|\SymGDev u\cdot\xi^D|^2+\left(\eta(\dvr u)+\eta'(\dvr u)\dvr u\right)(\tr \xi)^2.
\end{equation}

We deal with the notion of strong ellipticity of the operator $\mathcal{A}(\mathbb Du)$ which means that there exists $C_{el}(\mathbb Du)>0$ such that 
\begin{equation}\label{StrongEll}
\sum_{j,k,l,m=1}^da^{lm}_{jk}\xi_{jl}\xi_{km}\geq C_{el}(\mathbb Du)|\xi|^2\text{ holds for all }\xi\in\mathbb{R}^{d\times d}.
\end{equation}

Let investigate necessary conditions on $\mu$ and $\eta$ related to the strong ellipticity of $\mathcal{A}(\mathbb Du)$.
We first choose $g\in\mathbb{R}^d$ as an eigenvector of $\SymGDev u$ and $h\in\mathbb{R}^d$ being perpendicular to $g$. We consider the matrix $(\xi)_{jk}=\frac{1}{2}(g_jh_k-g_kh_j)$, $j,k=1,\ldots d$ and we obtain from \eqref{BilFormDet} combined with \eqref{StrongEll} 
\begin{equation*}
2\mu(|\SymGDev u|^2)|\xi|^2\geq C_{el}(\mathbb Du)|\xi|^2
\end{equation*}
as $\xi$ is obviously traceless and $\xi^D=\xi$ consequently. Hence one concludes that 
\begin{equation*}
2\mu(s)\geq C_{el}(\mathbb Du)\text{ for all }s\in[0,\|\SymGDev u\|^2_{L^\infty(\Omega)}].
\end{equation*}
As $u\in C^1(\overline\Omega)^d$ can be chosen arbitrarily, we get
\begin{equation}\label{MuCond1}
\mu(s)> 0\text{ for any }s\geq 0.
\end{equation}
In order to get the next condition involving $\mu$, we set $\xi=\SymGDev u$ in \eqref{StrongEll}, which together with \eqref{BilFormDet} imply 
\begin{equation*}
2\mu(|\SymGDev u|^2)|\SymGDev u|^2+4\mu'(|\SymGDev u|^2)|\SymGDev u|^2|\SymGDev u|^2\geq C_{el}(u)|\SymGDev u|^2.
\end{equation*}
Hence one concludes that 
\begin{equation*}
2\mu(s)+4\mu'(s)s\geq C_{el}(\mathbb Du),\text{ for any }s\in [0,\|\SymGDev u\|^2_{L^\infty(\Omega)}]
\end{equation*}
and finally
\begin{equation}\label{MuCond2}
\mu(s)+2\mu'(s)s> 0\text{ for any }s\geq 0
\end{equation}
by repeating the arguments leading to \eqref{MuCond1}. To determine a condition on $\eta$ we choose $\xi=\IdM$ in \eqref{StrongEll}. It follows that $(\IdM)^D=0$ and we obtain from \eqref{BilFormDet} that
\begin{equation*}
\left(\eta(\dvr u)+\eta'(\dvr u)\dvr u\right)d^2\geq C_{el}(\mathbb Du)d.
\end{equation*}
Repeating the arguments leading to \eqref{MuCond1}, \eqref{MuCond2} respectively, we get
\begin{equation*}
\eta(r)+\eta'(r)r\geq \frac{C_{el}(u)}{d}\text{ for all }r\in[-\|\dvr u\|_{L^\infty(\Omega)},\|\dvr u\|_{L^\infty(\Omega)}]
\end{equation*}
and then
\begin{equation}\label{LambdaCond}
\eta(r)+\eta'(r)r>0\text{ for all }r\in\mathbb{R}.
\end{equation}
We note that conditions \eqref{MuCond1}, \eqref{MuCond2} and \eqref{LambdaCond} are also sufficient for the ellipticity of $\mathcal{A}(\mathbb Du)$. If $\mu'(|\SymGDev u|^2)\geq 0$ these conditions ensure that 
\begin{equation}\label{AuxStepEl1}
\sum_{j,k,l,m = 1}^d a^{lm}_{jk}\xi_{jl}\xi_{km}\geq 2\mu(|\SymGDev u|^2)|\xi^D|^2+(\eta(\dvr u)+\eta'(\dvr u)\dvr u)(\tr\xi)^2.
\end{equation}
From the continuity of $\mu$ and the function $r\mapsto\eta(r)+\eta'(r)r$ and from the assumption $u\in C^1(\overline\Omega)^d$ we conclude the existence of $s_{min}\in[0,\|\SymGDev u\|^2_{L^\infty(\Omega)}]$ and $r_{min}\in[-\|\dvr u\|_{L^ \infty},\|\dvr u\|_{L^ \infty}]$ such that $\mu(s)\geq \mu(s_{min})>0$ for all $s\in[0,\|\SymGDev u\|^2_{L^\infty(\Omega)}]$ and $\eta(r)+\eta'(r)r\geq \eta(r_{min})+\eta'(r_{min})r_{min}>0$ for all $r\in[-\|\dvr u\|_{L^ \infty},\|\dvr u\|_{L^ \infty}]$. It follows from \eqref{AuxStepEl1} that
\begin{equation}\label{AuxStepEl2}
\sum_{j,k,l,m=1}^d a^{lm}_{jk}\xi_{jl}\xi_{km}\geq 2\mu(s_{min})|\xi^D|^2+(\eta(r_{min})+\eta'(r_{min})r_{min})(\tr\xi)^2.
\end{equation}
Moreover, since one can easily check that the mapping $\xi\mapsto \sqrt{|\xi^D|^2+(\tr \xi)^2}$ is an equivalent norm on $\mathbb{R}^{d\times d}$, the ellipticity of $\mathcal{A}(u)$ follows from \eqref{AuxStepEl2} by setting $C_{el}=\min\{2\mu(s_{min}),\eta(r_{min})+\eta'(r_{min})r_{min})\}$.

If $\mu'(|\SymGDev u|^2)< 0$ one applies the Cauchy--Schwarz inequality in \eqref{BilFormDet} to obtain
\begin{multline}\label{AuxStepEl3}
\sum_{j,k,l,m=1}^da^{lm}_{jk}\xi_{jl}\xi_{km}\geq\\ 2\mu(|\SymGDev u|^2)|\xi^D|^2+4\mu'(|\SymGDev u|^2)|\SymGDev u|^2|\xi^D|^2+\left(\eta(\dvr u)+\eta'(\dvr u)\dvr u\right)(\tr \xi)^2.
\end{multline}
The continuity of $s\mapsto \mu(s)+2\mu'(s)s$ and the assumption $u\in C^1(\overline\Omega)^d$ imply the existence of $\tilde s_{min}\in[0,\|\SymGDev u\|^2_{L^\infty(\Omega)}]$ such that $\mu(s)+2\mu'(s)s\geq \mu(\tilde s_{min})+2\mu'(\tilde s_{min})\tilde s_{min}>0$ for all $s\in[0,\|\SymGDev u\|^2_{L^\infty(\Omega)}]$. Going back to \eqref{AuxStepEl3} the ellipticity of $\mathcal{A}(\mathbb Du)$ follows by setting $C_{el}=\min\{\mu(\tilde s_{min})+2\mu'(\tilde s_{min})\tilde s_{min},\eta(r_{min})+\eta'(r_{min})r_{min}\}$.

\section{Local-in-time well posedness  of the compressible Navier-Stokes}
\label{Bfp}
In this section we present a proof of Theorem \ref{Thm:Main}. First, we transform system \eqref{SystemClass} into Lagrangian coordinates. Then a linearization of the transformed system is derived. Next, employing Theorem \ref{thm:regularita} we show that the solution operator to the linearized system is a contraction on a suitably chosen function space. Finally, by the Banach fixed point theorem we conclude the local-in-time well posedness of the transformed system and of \eqref{SystemClass} accordingly.
\subsection{System in Lagrangian coordinates}

We define the Lagrangian  coordinates as
\begin{equation}\label{LCoord}
\begin{split}
X_u(t,y) & = y + \int_0^t u(s,X_u(s,y)) \ {\rm d}s.
\end{split}
\end{equation}
A general function $f(t,x):\Omega\mapsto \mathbb R$ fulfills
\begin{equation*}
\begin{split}
\partial_t f(t, X_u(t,y)) &= \partial_1 f(t, X_u(t,y)) + \nabla_x f(t, X_u(t,y)) \cdot u\\
\nabla_y f(t,X_u(t,y)) & = \nabla_x f(t, X_u(t,y)) \nabla_y X_u (t,y)\\
\nabla_x f & = \nabla_y f (\nabla_y X_u)^{-1}.
\end{split}
\end{equation*}
The Jacobi matrix of the transformation $X_u$ is $\mathbb I_d+\int_0^t\nabla_y u(s,X_u(s,y))\ds$. We tacitly assume the invertibility of this matrix, which is ensured if 
\begin{equation}\label{SigmaCond}
\sup_{t\in(0,T)}\left\|\int_0^t\nabla_y u(s,\cdot)\right\|_{L^\infty(\Omega)}\ds< \sigma
\end{equation}
for some small number $\sigma$. Since we will work with functions that possess Lipschitz regularity with respect to space variables, the latter condition is fulfilled if $T$ is chosen suitably small.
In what follows we use the notation $E_u:= \mathbb I_d - (\nabla X_u)^{-1}$. Let us note that we get from \eqref{LCoord}
\begin{equation}\label{EExpr}
E_u(t,y)=W\left(\int_0^t\nabla u(s,X_u(s,y))\ds\right),
\end{equation}
where $W$ is a $d\times d$--matrix valued mapping that is smooth with respect to matrices $B\in\mathbb{R}^{d\times d}$ with $|B|<2\sigma$ and $W(0)=0$.
We define $\tilde \varrho (t,y) = \varrho(t,X(t,y))$ and $\tilde u (t,y)= u(t,X(t,y))$ and we rewrite \eqref{SystemClass} as
\begin{equation}\label{eq:NS.Lagrange}
\begin{split}
\pat \tilde\varrho + \tilde\varrho \dvr_y \tilde u &= G(\tilde \varrho, \tilde u)\\
\tilde\varrho \pat \tilde u  + \nabla_y \pi(\tilde \varrho) - \dvr_y \CrlS(\mathbb D_y \tilde u) &= F(\tilde \varrho,\tilde u)
\end{split}
\end{equation}
where 
\begin{equation*}
\begin{split}
G(\tilde \varrho,\tilde u ) &= - \nabla \tilde u \cdot E_{\tilde u}\tilde \varrho\\
F(\tilde \varrho, \tilde u) & = \nabla_y \pi(\tilde\varrho) E_{\tilde u} + \dvr_y\tilde{\mathcal{S}}(\tilde {u}) - \dvr_y \mathcal{S}(\SymGDev_y\tilde u) - \tr(\nabla_y\tilde {\mathcal S}(\tilde u)
E_{\tilde u})\\
\mathcal{S}(\SymGDev_y\tilde u)&=\mu(|\SymGDev_y\tilde u|^2)\SymGDev_y\tilde u+\eta(\dvr_y \tilde u)\dvr_y\tilde u\IdM\\
\tilde{\mathcal{S}}(\tilde u)&=\mu(|\mathcal G(\tilde u)|^2)\mathcal G(\tilde u)+\eta(\mathcal{D}(\tilde u))\mathcal D(\tilde u)\IdM\\
\mathcal{G}(\tilde u)&=\SymGDev_y\tilde u - \sym (\nabla_y \tilde u  E_{\tilde u})^{D},\ \mathcal D(\tilde u)=\dvr_y \tilde u-\nabla_y \tilde u\cdot E_{\tilde u}.
\end{split}
\end{equation*}

\subsection{Linearized system}
Let $\tilde\varrho = \varrho_* + \theta_0 + \theta$ where $\varrho_*$ is a constant and $\theta_0 = \varrho_0 - \varrho_*$ is independent of $t$. Recall both $\varrho_*$ and $\theta_0$ are given. Consequently, \eqref{eq:NS.Lagrange} becomes (recall \eqref{eq:def.A})
\begin{equation}
\begin{split}\label{eq:linearized}
\pat \theta + (\varrho_* + \theta_0) \dvr_y (\tilde u  - u_0)& = \mathscr G (\theta,\tilde u)\\
(\varrho_* + \theta_0) \pat \tilde u - \mathcal A(\mathbb D_y u_0)(\mathbb D_y(\tilde u-u_0)) + \pi'(\varrho_* + \theta_0)\nabla_y \theta & = \mathscr F(\theta,\tilde u)
\end{split}
\end{equation}
where 
\begin{equation}\label{RHSExpr}
\begin{split}
\mathscr G (\theta,\tilde u)  =& G(\varrho_* + \theta_0 + \theta,\tilde u) - \theta \dvr_y \tilde u - (\varrho_* + \theta_0)\dvr_y u_0\\
\mathscr F (\theta,\tilde u)  =& F(\varrho_* + \theta_0 + \theta, \tilde u) - \theta \tder \tilde u + \dvr_y \mathcal{S}(\SymGDev_y\tilde u) - \mathcal A(\mathbb D_yu_0)(\mathbb D_y(\tilde u-u_0)) + \pi'(\varrho_* + \theta_0) \nabla_y \theta_0  \\&
+ \left(\pi'(\varrho_* + \theta_0 + \theta) - \pi'(\varrho_* + \theta_0)\right)\nabla_y(\theta_0 + \theta)
\end{split}
\end{equation}
and the system is equipped with the initial conditions $\theta(0,\cdot)=0$, $\tilde u(0,\cdot)=u_0$.
In order to shorten the notation we omit the subscript $y$ in the rest of this section.

\subsection{The fix-point argument}
We define
\begin{align*}
H_{T,M}=\{&(\vartheta,w)\in W^{1,p}(0,T;W^{1,q}(\Omega))\times \mathcal V^{p,q}(Q_T):\
\vartheta(0)=0, w(0)=u_0,\ \\
&[\vartheta,w]_{T,M}=\|\vartheta\|_{W^{1,p}(0,T;W^{1,q}(\Omega))}+\|w-u_0\|_{\mathcal{V}^{p,q}}\leq M\}
\end{align*}
and an operator $\Phi$ as
\begin{equation*}
\Phi(\vartheta,w) = (\theta, \tilde u)
\end{equation*}
where $\theta\in W^{1,p}(0,T;W^{1,q}(\Omega))$, $\tilde u\in L^p(0,T;W^{2,q}(\Omega)^d)\cap W^{1,p}(0,T;L^q(\Omega)^d)$ is a solution to 
\begin{equation*}
\begin{split}
\pat \theta + (\varrho_* + \theta_0) \dvr (\tilde u -u_0)&= \mathscr G(\vartheta,w)\\
(\varrho_* + \theta_0)\pat \tilde u - \mathcal A(\mathbb Du_0)(\mathbb D(\tilde u-u_0)) + \pi'(\varrho_* + \theta_0)\nabla\theta & = \mathscr F(\vartheta,w)\\
\theta(0)=0,\ \tilde u(0)=u_0.
\end{split}
\end{equation*}

In order to apply the Banach fixed-point theorem, we need to justify that for certain values of $T$ and $M$ the operator $\Phi$ is a contraction on $H_{T,M}$. 
Next we introduce several estimates. We recall the embedding inequality
\begin{equation}\label{EmbIneq}
\sup_{t\in[0,T]}\|u\|_{W^{1,\infty}(\Omega)}\leq c\left(\|u\|^p_{L^p(0,T;W^{2,q}(\Omega))}+\|u\|^p_{W^{1,p}(0,T;L^q(\Omega))}\right)^\frac{1}{p},
\end{equation}
see \cite[p. 10]{Ta} for details, which is valid for $u\in\mathcal V^{p,q}(Q_T)$, $p\in\left(\frac{2q}{q-d},\infty\right)$ and $q>d$.  It follows that
\begin{equation*} 
\|w\|_{\mathcal V^{p,q}(Q_T))}\in[T^\frac{1}{p}\|u_0\|_{W^{2,q}(\Omega)}-M,T^\frac{1}{p}\|u_0\|_{W^{2,q}(\Omega)}+M].
\end{equation*} 
for an arbitrary $w\in B_M=\{v\in\mathcal V^{p,q}(Q_T):v(0)=u_0, \|v-u_0\|_{\mathcal V^{p,q}(Q_T)}\leq M\}$.
Accordingly, we have $\|w\|_{\mathcal V^{p,q}(Q_T)}\leq M+LT^\frac{1}{p}$ where \begin{equation*}
L=\|\theta_0\|_{W^{1,q}(\Omega)}+\|u_0\|_{W^{2,q}(\Omega)}.
\end{equation*}
Employing \eqref{EmbIneq} we deduce that 
\begin{equation}\label{SymGDivBound}
\|\SymGDev w\|_{L^\infty(0,T;L^\infty(\Omega))}+\|\dvr w\|_{L^\infty(0,T;L^\infty(\Omega))}\leq c\|w\|_{\mathcal V^{p,q}(Q_T)}\leq c(M+LT^\frac{1}{p}).
\end{equation}
By the Sobolev embedding theorem and the H\"older inequality we infer 
\begin{equation*}
\int_0^T\|\nabla w(s,\cdot)\|_{L^\infty(\Omega)}\ds\leq \int_0^T\|w(s,\cdot)\|_{W^{2,q}(\Omega)}\ds\leq cT^\frac{1}{p'}\|w\|_{L^p(0,T;W^{2,q}(\Omega))}.
\end{equation*}
Hence  the condition \eqref{SigmaCond} is fulfilled for any $w\in H_{T,M}$ assuming $T$ is suitably small.
Moreover, it follows from \eqref{EExpr} and the paragraph below it that 
\begin{equation}\label{EEst}
\|E_w(t)\|_{W^{k,q}(\Omega)}\leq c\left\|\int_0^t\nabla w\ {\rm d}s\right\|_{W^{k,q}(\Omega)}\leq c\int_0^t\|w\|_{W^{k+1,q}(\Omega)}\ {\rm d}s,\ k=0,1.
\end{equation}
The Sobolev embedding then yields (since $q>d$)
\begin{equation}\label{SupEEst}
\|E_w(t)\|_{L^\infty(0,T;L^\infty(\Omega))}\leq cT^\frac{1}{p'}\|w\|_{L^p(0,T;W^{2,q}(\Omega))}
\end{equation}
and we use H\"older inequality to deduce
\begin{equation}\label{ELPQEst}
\|E_w\|_{L^p(0,T;W^{1,q}(\Omega))}\leq c\left(\int_0^Tt^\frac{p}{p'}\|\nabla w\|^p_{L^p(0,T;W^{1,q}(\Omega))}\ {\rm d}t\right)^\frac{1}{p}\leq cT\|w\|_{L^p(0,T;W^{2,q}(\Omega))}.
\end{equation}
By \eqref{SupEEst} and \eqref{EmbIneq} we infer
\begin{align*}
\|\mathcal{G}(w)\|_{L^\infty(0,T;L^\infty(\Omega))}\leq &c\|\nabla w\|_{L^\infty(0,T;L^\infty(\Omega))}(1+\|E_{w}\|_{L^\infty(0,T;L^\infty(\Omega))})\\
\leq &c\|w\|_{\mathcal V^{p,q}(Q_T)}(1+T^\frac{1}{p'}\|w\|_{L^p(0,T;W^{2,q}(\Omega))})
\end{align*}
and thus
\begin{equation}\label{CrlGEst}
\|\mathcal{G}(w)\|_{L^\infty(0,T;L^\infty(\Omega))}\leq c(M+LT^\frac{1}{p})(1+T^\frac{1}{p'}(M+LT^\frac{1}{p}))
\end{equation}
due to \eqref{SymGDivBound}. One similarly deduces
\begin{equation}\label{CrlDEst}
\|\mathcal{D}(w)\|_{L^\infty(0,T;L^\infty(\Omega))}\leq c(M+LT^\frac{1}{p})(1+T^\frac{1}{p'}(M+LT^\frac{1}{p})).
\end{equation}
From now on we restrict ourselves to $T\leq 1$ and fix $R>c(1+L^2)$, where $c$ is the maximum of constants from \eqref{SymGDivBound}, \eqref{CrlGEst} and \eqref{CrlDEst}. Moreover, we assume that $M$ fulfills $R\geq c(1+L+M)^2$ and define the quantity $K$ as follows
\begin{align*}
K=\sup_{s\in[0,R^2]}&\sum_{j=0}^3|\mu^{(j)}(s)|+\sup_{t\in[-R,R]}\sum_{j=0}^2|\eta^{(j)}(t)|.
\end{align*}
Obviously, the assumed smoothness of functions $\mu$ and $\eta$ ensures that $K$ is finite.
Furthermore, we define
\begin{equation*}
\Pi=\sup_{r\in[-\tilde cM,\tilde c M]}  \sum_{j=0}^2\|\pi^{(j)}(\varrho_*+\theta_0+r)\|_{L^\infty(Q_T)}
\end{equation*}
where $\tilde c$ is such that $\|\theta\|_{L^\infty(0,T;L^\infty(\Omega))}\leq \tilde c M$ for all $(\theta, w)\in H_{T,M}$ (see also \eqref{ThetaEst}).
The fact that $\Pi$ is finite follows as $\varrho_*+\theta_0$ is a continuous function by the Sobolev embedding and the assumption $\varrho_*+\theta_0\in W^{1,q}(\Omega)$, $q>d$.

First, we check that $\Phi:H_{T,M}\to H_{T,M}$ if $T$ and $M$ are chosen appropriately. To this end we fix $(\vartheta,w)\in H_{T,M}$ and estimate $\mathscr G (\vartheta,w)$, $\mathscr F (\vartheta,w)$ respectively. 
As $\vartheta(0)=0$, we get for $t\in[0,T]$ using the H\"older inequality
\begin{equation*}
\|\vartheta(t)\|_{W^{1,q}(\Omega)}\leq \int_0^t\|\tder\vartheta\|_{W^{1,q}(\Omega)}\ {\rm d}s\leq T^\frac{1}{p'}\|\tder\vartheta\|_{L^p(0,T;W^{1,q}(\Omega))}
\end{equation*}
and consequently
\begin{equation}\label{ThetaEst}
\|\vartheta\|_{L^\infty(0,T;W^{1,q}(\Omega))}\leq T^\frac{1}{p'}\|\vartheta\|_{W^{1,p}(0,T;W^{1,q}(\Omega))}.
\end{equation}
Let us focus on the estimate of $\mathscr G (\vartheta,w)$. Using the Sobolev embedding, \eqref{EEst}, \eqref{ThetaEst} and the H\"older inequality it follows that
\begin{equation}\label{RhsGEst}
\begin{split}
\|\nabla w\cdot E_w(\varrho_*+\theta_0+\vartheta)\|_{W^{1,q}(\Omega)}&\leq c\|w\|_{W^{2,q}(\Omega)}\|E_w\|_{W^{1,q}}\|\varrho_*+\theta_0+\vartheta\|_{W^{1,q}(\Omega)}\\
&\leq c \|w\|_{W^{2,q}(\Omega)}T^\frac{1}{p'}\|w\|_{L^p(0,T;W^{2,q}(\Omega))}\left(\|\varrho_*+\theta_0\|_{W^{1,q}(\Omega)}+\|\vartheta\|_{W^{1,q}(\Omega)}\right)\\
&\leq c(M+L)^2T^\frac{1}{p'}\|w\|_{W^{2,q}(\Omega)}
\end{split}
\end{equation}
and
\begin{equation}\label{RhsGEst2}
\|\vartheta \dvr w\|_{W^{1,q}(\Omega)}\leq c\|\vartheta\|_{W^{1,q}(\Omega)}\|w\|_{W^{2,q}(\Omega)}\leq cT^\frac{1}{p'}M\|w\|_{W^{2,q}(\Omega)}.
\end{equation}
In order to proceed with the bound of $\mathscr F(\vartheta, w)$ we begin with some preparatory work. The assumed smoothness of $\mu$ and $\eta$ yields by the mean value theorem
\begin{equation}\label{MuDiff}
\begin{split}
\mu^{(k)}(|A|^2)-\mu^{(k)}(|B|^2)=&\int_0^1 \frac{\dd}{\ds}\mu^{(k)}(|sA+(1-s)B|^2)\ds\\
=&2\int_0^1 \mu^{(k+1)}(|sA+(1-s)B|^2)(sA+(1-s)B)\cdot(A-B)\ds\ k=0,1,2,
\end{split}
\end{equation}
\begin{equation}\label{MuPrDiff}
\begin{split}
&\mu'(|A|^2)|A|^2-\mu'(|B|^2)|B|^2=\int_0^1 \frac{\dd}{\ds}\left(\mu'(|sA+(1-s)B|^2)|sA+(1-s)B|^2\right)\ds\\
&=2\int_0^1 (sA+(1-s)B)\cdot(A-B)\left(\mu''(|sA+(1-s)B|^2)|sA+(1-s)B|^2+\mu'(|sA+(1-s)B|^2)\right)\ds,
\end{split}
\end{equation}
\begin{equation}\label{MuDivDiff}
\begin{split}
\mu(|A|^2)\dvr A-\mu(|B|^2)\dvr B=&\int_0^1 \frac{\dd}{\ds} \left(\mu(|sA+(1-s)B|^2)\dvr (sA+(1-s)B)\right)\ds\\
=&\int_0^1 2\mu'(|sA+(1-s)B|^2)(sA+(1-s)B)\cdot(A-B)\dvr (sA+(1-s)B)\\&+\mu(|sA+(1-s)B|^2)\dvr(A-B)\ds\\
\end{split}
\end{equation}
and
\begin{equation}\label{MuPrDivDiff}
\begin{split}
\mu'(&|A|^2)|A|^2\dvr A-\mu'(|B|^2)|B|^2\dvr B\\
=&\int_0^1 \frac{\dd}{\ds} \left(\mu'(|sA+(1-s)B|^2)|sA+(1-s)B|^2\dvr (sA+(1-s)B)\right)\ds\\
=&\int_0^1 2(sA+(1-s)B)\cdot(A-B)\left(\mu''(|sA+(1-s)B|^2)|sA+(1-s)B|^2\right.\\
&\left.+\mu'(|sA+(1-s)B|^2)\right)\dvr (sA+(1-s)B)\\
&+\mu'(|sA+(1-s)B|^2)|sA+(1-s)B|^2\dvr(A-B)\ds
\end{split}
\end{equation}
assuming $A,B$ are $d\times d$--matrix valued sufficiently smooth functions.
Moreover,  it follows that
\begin{equation}\label{LambdaDiff}
\eta(a)-\eta(b)=\int_0^1 \eta'(sa+(1-s)b)(a-b)\ds,
\end{equation}
\begin{equation*}
\begin{split}
\eta'(a)a-\eta'(b)b=&\int_0^1 \frac{\dd}{\ds}\left(\eta'(sa+(1-s)b)(sa+(1-s)b)\right)\ds\\=&\int_0^1\left( \eta'(sa+(1-s)b)+\eta''(sa+(1-s)b)(sa+(1-s)b)\right)(a-b)\ds,
\end{split}
\end{equation*}
\begin{equation}\label{LambdaNabDiff1}
\begin{split}
\eta(a)\nabla a-\eta(b)\nabla b=&\int_0^1 \frac{\dd}{\ds}\left(\eta(sa+(1-s)b)\nabla (sa+(1-s)b)\right)\ds\\
=& \int_0^1\eta'(sa+(1-s)b)(a-b)\nabla (sa+(1-s)b) + \eta(sa+(1-s)b)\nabla(a-b)\ds
\end{split}
\end{equation}
\begin{equation}\label{LambdaNabDiff2}
\begin{split}
\eta'(a)a\nabla a-\eta'(b)b\nabla b=&\int_0^1 \frac{\dd}{\ds}\left(\eta'(sa+(1-s)b)(sa+(1-s)b)\nabla (sa+(1-s)b)\right)\ds\\
=&\int_0^1 \eta''(sa+(1-s)b)(a-b)(sa+(1-s)b)\nabla (sa+(1-s)b)\\
&+\eta'(sa+(1-s)b)(a-b)\nabla (sa+(1-s)b)\\
&+\eta'(sa+(1-s)b)(sa+(1-s)b)\nabla (a-b)\ds
\end{split}
\end{equation}
for arbitrary real valued functions $a,b$.

Recalling \eqref{eq:ACoef} we get using \eqref{MuDiff}, \eqref{MuPrDiff}, \eqref{LambdaDiff}, \eqref{LambdaNabDiff1}, \eqref{LambdaNabDiff2}, \eqref{EmbIneq} and the Jensen inequality
\begin{equation}\label{ACoefDifEst}
\begin{split}
&\|a_{jk}^{lm}(\mathbb Du)-a_{jk}^{lm}(\mathbb Dv)\|_{L^\infty(0,T;L^\infty(\Omega))}\\&\leq K\int_0^1\left(1+\|s\SymGDev u+(1-s)\SymGDev v\|^2_{L^\infty(0,T;L^\infty(\Omega))}\right)\|s\SymGDev u+(1-s)\SymGDev v\|_{L^\infty(0,T;L^\infty(\Omega))}\\
&\times\|\SymGDev (u-v)\|_{L^\infty(0,T;L^\infty(\Omega))}+(1+\|s\dvr u+(1-s)\dvr v\|_{L^\infty(0,T;L^\infty(\Omega))})\|\dvr(u-v)\|_{L^\infty(0,T;L^\infty(\Omega))}\ds\\
&\leq cK(1+\|u\|_{\mathcal V^{p,q}(Q_T)}+\|v\|_{\mathcal V^{p,q}(Q_T)}+\|u\|^3_{\mathcal V^{p,q}(Q_T)}+\|v\|^3_{\mathcal V^{p,q}(Q_T)})\|u-v\|_{\mathcal V^{p,q}(Q_T)},
\end{split}
\end{equation}
whenever $u,v\in \mathcal V^{p,q}(Q_T)$.

We can proceed with the bound of \eqref{RHSExpr}$_2$. First, we estimate using \eqref{CrlGEst}, \eqref{CrlDEst} and \eqref{SupEEst}
\begin{equation}\label{TrSEEst}
\begin{split}
\|\tr(\nabla \tilde{\mathcal{S}}(w)E_{w})\|_{L^p(0,T;L^q(\Omega))}\leq&K \left((1+\|\mathcal G(w)\|_{L^\infty(0,T;L^\infty(\Omega))})\|\nabla \mathcal G(w)\|_{L^p(0,T;L^q(\Omega))}\right.\\
&\left.+(1+\|\mathcal D(w)\|_{L^\infty(0,T;L^\infty(\Omega))})\|\nabla\mathcal D(w)\|_{L^p(0,T;L^q(\Omega))}\right)\|E_{w}\|_{L^\infty(0,T;L^\infty(\Omega))}\\
\leq& cKT^\frac{1}{p'}(M+L)^2(1+(M+L)(1+M+L))(1+M+L).
\end{split}
\end{equation}
as it follows that
\begin{equation*}
\begin{split}
\|\nabla\mathcal G(w)\|_{L^p(0,T;L^q(\Omega))}\leq& \|w\|_{L^p(0,T;W^{2,q}(\Omega))}(1+\|E_w\|_{L^\infty(0,T;L^\infty(\Omega))})\\
&+\|w\|_{L^\infty(0,T;L^\infty(\Omega))}\|\nabla E_w\|_{L^p(0,T;L^q(\Omega))}\\
\leq& c\|w\|_{\mathcal V^{p,q}(Q_T)}(1+\|w\|_{\mathcal V^{p,q}(Q_T)})\leq c(M+L)(1+M+L)
\end{split}
\end{equation*}
and
\begin{equation*}
\begin{split}
\|\nabla\mathcal D(w)\|_{L^p(0,T;L^q(\Omega))}\leq& \|w\|_{L^p(0,T;W^{2,q}(\Omega))}(1+\|E_w\|_{L^\infty(0,T;L^\infty(\Omega))})\\
&+\|w\|_{L^\infty(0,T;L^\infty(\Omega))}\|\nabla E_w\|_{L^p(0,T;L^q(\Omega))}\\
\leq& c\|w\|_{\mathcal V^{p,q}(Q_T)}(1+\|w\|_{\mathcal V^{p,q}(Q_T)})\leq c(M+L)(1+M+L)
\end{split}
\end{equation*}
by \eqref{SupEEst} and \eqref{ELPQEst}.

Then we continue by expanding the divergence to arrive at
\begin{equation*}
\begin{split}
&\dvr(\tilde{\mathcal S}(w)-\mathcal{S}(\SymGDev w))=\mu(|\mathcal G(w)|^2)\dvr\mathcal G(w)-\mu(|\SymGDev w|^2)\dvr\SymGDev w\\&+2\mu'(|\mathcal G(w)|^2)|\mathcal G(w)|^2\dvr\mathcal G(w)-2\mu(|\SymGDev w|^2)|\SymGDev w|^2\dvr\SymGDev w\\
&+\eta(\mathcal D(w))\nabla\mathcal D(w)-\eta(\dvr w)\nabla\dvr w\\
&+\eta'(\mathcal D(w))\mathcal D(w)\nabla\mathcal D(w)-\eta'(\dvr w)\dvr w\nabla\dvr w=\sum_{i=1}^4I_j.
\end{split}
\end{equation*}
We estimate each $I_j$ separately. We observe that denoting 
\begin{equation*}
\mathcal G_s=s(\SymGDev w-\sym(\nabla wE_w)^D)+(1-s)\SymGDev w
\end{equation*}
it follows by \eqref{MuDivDiff} that
\begin{align*}
I_1=-\int_0^1 2\mu'(|\mathcal G_s|^2)\mathcal G_s\cdot\sym(\nabla wE_w)^D\dvr\mathcal G_s+\mu(|\mathcal G_s|^2)\dvr\sym(\nabla wE_w)^D\ds.
\end{align*}
Further we employ the estimates
\begin{equation}\label{GSEst}
\begin{split}
\|\mathcal G_s\|_{L^\infty(0,T;L^\infty(\Omega))}\leq& c\|\SymGDev w\|_{L^\infty(0,T;L^\infty(\Omega))}+\|\nabla w\|_{L^\infty(0,T;L^\infty(\Omega))}\|E_w\|_{L^\infty(0,T;L^\infty(\Omega))}
\\
\leq& c\|w\|_{\mathcal V^{p,q}(Q_T)}(1+T^\frac{1}{p'}\|w\|_{\mathcal{V}^{p,q}(Q_T)}),
\end{split}
\end{equation}
\begin{equation}\label{DivGSEst}
\begin{split}
\|\dvr\mathcal G_s \|_{L^p(0,T;L^q(\Omega))}\leq& c\|w\|_{L^p(0,T;L^q(\Omega))}(1+\|E_{w}\|_{L^\infty(0,T;L^\infty(\Omega))})+\|\nabla w\|_{L^\infty(0,T;L^\infty(\Omega))}\|\nabla E_w\|_{L^p(0,T;L^q(\Omega))}\\
\leq&c\|w\|_{\mathcal V^{p,q}(Q_T)}(1+(T^\frac{1}{p'}+T)\|w\|_{\mathcal V^{p,q}(Q_T)}),
\end{split}
\end{equation}
\begin{equation}\label{DivSymEst}
\begin{split}
\|\dvr\sym(\nabla wE_w)^D\|_{L^p(0,T;L^q(\Omega))}\leq&c\|\nabla^2 w\|_{L^p(0,T;L^q(\Omega))}\|E_w\|_{L^\infty(0,T;L^\infty(\Omega))}\\
&+\|\nabla w\|_{L^\infty(0,T;L^\infty(\Omega))}\|\nabla E_w\|_{L^p(0,T;L^q(\Omega))}\\
\leq&c\|w\|^2_{\mathcal V^{p,q}(Q_T)}(T^\frac{1}{p'}+T)
\end{split}
\end{equation}
that follow by \eqref{EEst}, \eqref{ELPQEst} and \eqref{EmbIneq}.
Employing the Jensen inequality, \eqref{GSEst}, \eqref{DivGSEst} and \eqref{DivSymEst} we infer
\begin{equation}\label{I1Est}
\begin{split}
\|&I_1\|_{L^p(0,T;L^q(\Omega))}\\
\leq&cK\|w\|^2_{\mathcal V^{p,q}(Q_T)}(1+T^\frac{1}{p'}\|w\|_{\mathcal V^{p,q}(Q_T)})\|\nabla w\|_{L^\infty(0,T;L^\infty(\Omega))}\|E_w\|_{L^\infty(0,T;L^\infty(\Omega))}(1+(T^\frac{1}{p'}+T)\|w\|_{\mathcal V^{p,q}(Q_T)})\\
&+cK\|w\|^2_{\mathcal V^{p,q}(Q_T)}(T^\frac{1}{p'}+T)\\
\leq& cK \|w\|^2_{\mathcal V^{p,q}(Q_T)}\left((1+T^\frac{1}{p'}\|w\|_{\mathcal V^{p,q}(Q_T)})T^\frac{1}{p'}\|w\|^2_{\mathcal V^{p,q}(Q_T)} (1+(T^\frac{1}{p'}+T)\|w\|_{\mathcal V^{p,q}(Q_T)})+T^\frac{1}{p'}+T\right)\\
\leq& cK(M+L)^2\left((1+(M+L))^2T^\frac{1}{p'}(M+L)^2+T^\frac{1}{p'}+T\right),
\end{split}
\end{equation}
where the Sobolev embedding, interpolation inequality \eqref{EmbIneq} and \eqref{ELPQEst} were also employed.
By \eqref{MuPrDivDiff} we have
\begin{multline*}
I_2 =-\int_0^1 2\left(\mu''(|\mathcal G_s|^2)|\mathcal G_s|^2+\mu'(|\mathcal G_s|^2)\right)\mathcal G_s\cdot\sym(\nabla wE_w)^D\dvr\mathcal G_s\\+\mu(|\mathcal G_s|^2)|\mathcal G_s|^2\dvr\sym(\nabla wE_w)^D\ds.
\end{multline*}
Then using  \eqref{GSEst}, \eqref{DivGSEst} and \eqref{DivSymEst} we obtain
\begin{equation}\label{I2Est}
\begin{split}
\|&I_2\|_{L^p(0,T;L^q(\Omega))}\\
\leq &cK\left((\|w\|^2_{\mathcal V^{p,q}(Q_T)}(1+T^\frac{1}{p'}\|w\|_{\mathcal V^{p,q}(Q_T)})^2+1)\|w\|_{\mathcal V^{p,q}(Q_T)}(1+T^\frac{1}{p'}\|w\|_{\mathcal V^{p,q}(Q_T)})\|\nabla w\|_{L^\infty(0,T;L^\infty(\Omega))}\right.\\
&\left.\times \|E_w\|_{L^\infty(0,T;L^\infty(\Omega))}\|w\|_{\mathcal V^{p,q}(Q_T)}(1+(T^\frac{1}{p'}+T)\|w\|_{\mathcal V^{p,q}(Q_T)})+\|w\|^4_{\mathcal V^{p,q}(Q_T)}(1+T^\frac{1}{p'}\|w\|_{\mathcal V^{p,q}(Q_T)})(T^\frac{1}{p'}+T)\right)\\
\leq& cK\left(((M+L)^2(1+(M+L)^2)+1)(M+L)^4(1+(M+L))T^\frac{1}{p'}(1+(M+L))\right.\\
&\left.+(M+L)^4(1+(M+L))(T^\frac{1}{p'}+T)\right).
\end{split}
\end{equation}
Denoting 
\begin{equation*}
\mathcal D_s=s(\dvr w-\nabla w\cdot E_w)+(1-s)\dvr w
\end{equation*}
we obtain by \eqref{LambdaNabDiff1}
\begin{align*}
I_3=-\int_0^1 \eta'(\mathcal D_s)\nabla w\cdot E_w\nabla\mathcal D_s+\eta(\mathcal D_s)\nabla(\nabla w\cdot E_w)\ {\rm d}s.
\end{align*}
Moreover, one has by \eqref{EmbIneq} and \eqref{EEst}
\begin{equation}\label{DSEst}
\begin{split}
\|\mathcal D_s\|_{L^\infty(0,T;L^\infty(\Omega))}\leq& c\|\nabla w\|_{L^\infty(0,T;L^\infty(\Omega))}(1+\|E_w\|_{L^\infty(0,T;L^\infty(\Omega))})\\
\leq& c\|w\|_{\mathcal V^{p,q}(Q_T)}(1+T^\frac{1}{p'}\|w\|_{\mathcal V^{p,q}(Q_T)})
\end{split}
\end{equation}
and
\begin{equation}\label{NabDSEst}
\begin{split}
\|\nabla\mathcal D_s\|_{L^p(0,T;L^q(\Omega))}\leq& c\|\nabla^2w\|_{L^p(0,T;W^{2,q}(\Omega))}+\|\nabla^2w\|_{L^p(0,T;W^{2,q}(\Omega))}\|E_w\|_{L^\infty(0,T;L^\infty(\Omega))}\\
&+\|\nabla w\|_{L^\infty(0,T;L^\infty(\Omega))}\|\nabla E_w\|_{L^p(0,T;L^q(\Omega))}\\
\leq& c\|w\|_{\mathcal V^{p,q}(Q_T)}\left(1+(T^\frac{1}{p'}+T)\|w\|_{\mathcal V^{p,q}(Q_T)}\right).
\end{split}
\end{equation}
Then it follows by \eqref{EmbIneq}, \eqref{EEst}, \eqref{DSEst} and \eqref{NabDSEst}
\begin{equation}\label{I3Est}
\begin{split}
\|I_3\|_{L^p(0,T;L^q(\Omega))}\leq &cK\left(\|\nabla w\|_{L^\infty(0,T;L^\infty(\Omega))}\|E_w\|_{L^\infty(0,T;L^\infty(\Omega))}\|\nabla \mathcal D_s\|_{L^p(0,T;L^q(\Omega))}\right.\\
&\left.+\|\nabla^2 w\|_{L^p(0,T;L^q(\Omega))}\|E_w\|_{L^\infty(0,T;L^\infty(\Omega))}+\|\nabla w\|_{L^\infty(0,T;L^\infty(\Omega))}\|\nabla E_w\|_{L^p(0,T;L^q(\Omega))}\right)\\
\leq &cK\left(T^\frac{1}{p'}\|w\|^3_{\mathcal V^{p,q}(Q_T)}\left(1+(T^\frac{1}{p'}+T)\|w\|_{\mathcal V^{p,q}(Q_T)}\right)+(T^\frac{1}{p'}+T)\|w\|^2_{\mathcal V^{p,q}(Q_T)}\right)\\
\leq&cK\left(T^\frac{1}{p'}(M+L)^3\left(1+(M+L)\right)+(T^\frac{1}{p}+T)(M+L)^2\right).
\end{split}
\end{equation}
By \eqref{LambdaNabDiff2} we obtain
\begin{align*}
I_4=&-\int_0^1 \eta''(\mathcal D_s)\nabla w\cdot E_w\mathcal D_s\nabla\mathcal D_s+\eta'(\mathcal D_s)\nabla w\cdot E_w\nabla\mathcal D_s\\
&+\eta'(\mathcal D_s)\mathcal D_s\nabla(\nabla w\cdot E_w)\ds.
\end{align*}
Finally, we get by \eqref{CrlDEst} and \eqref{EmbIneq}
\begin{equation}\label{I4Est}
\begin{split}
\|&I_4\|_{L^p(0,T;L^q(\Omega))}\\
\leq& cK \left(\|\nabla w\|_{L^\infty(0,T;L^\infty(\Omega))}\|E_{w}\|_{L^\infty(0,T;L^\infty(\Omega))}(1+\|\mathcal D_s\|_{L^\infty(0,T;L^\infty(\Omega))})\|\nabla\mathcal D_s\|_{L^p(0,T;L^q(\Omega))}\right.\\
&\left.+\|\mathcal D_s\|_{L^\infty(0,T;L^\infty(\Omega))}(\|\nabla^2w\|_{L^p(0,T;L^q(\Omega))}\|E_w\|_{L^\infty(0,T;L^\infty(\Omega))}+\|\nabla w\|_{L^\infty(0,T;L^\infty(\Omega))}\|\nabla E_w\|_{L^p(0,T;L^q(\Omega))}\right)\\
\leq&cK\left(T^\frac{1}{p'}\|w\|^4_{\mathcal V^{p,q}(Q_T)}(1+T^\frac{1}{p'}\|w\|_{\mathcal V^{p,q}(Q_T)})\left(1+(T^\frac{1}{p'}+T)\|w\|_{\mathcal V^{p,q}(Q_T)}\right)\right.\\
&\left.+(T^\frac{1}{p'}+T)\|w\|^3_{\mathcal V^{p,q}(Q_T)}(1+T^\frac{1}{p'}\|w\|_{\mathcal V^{p,q}(Q_T)})\right)\\
\leq&cK\left(T^\frac{1}{p'}(M+L)^4\left(1+(M+L)\right)^2+(T^\frac{1}{p'}+T)(1+M+L)^3\right).
\end{split}
\end{equation}
In order to proceed with the estimate of $\mathscr{F}$ we have
\begin{equation}\label{PrDiffEst2}
\begin{split}
&\|\pi'(\varrho_*+\theta_0+\vartheta)\nabla(\theta_0+\vartheta)E_w\|_{L^p(0,T;L^q(\Omega))}\leq \Pi\|\nabla(\theta_0+\vartheta)\|_{L^p(0,T;L^q(\Omega))}\|E_{w}\|_{L^\infty(0,T;L^\infty(\Omega))}\\
&\leq c\Pi(L+M)T^\frac{1}{p'}(M+L).
\end{split}
\end{equation}
Applying \eqref{ThetaEst} and the Sobolev embedding we obtain
\begin{equation}\label{TDerEst}
\|\vartheta\tder w\|_{L^p(0,T;L^q(\Omega))}\leq \|\vartheta\|_{L^\infty(0,T;L^\infty(\Omega))}\|\tder w\|_{L^p(0,T;L^q(\Omega))}\leq c T^\frac{1}{p'}M(M+L).
\end{equation}
One immediately has
\begin{equation}\label{PrDiffEst3}
\|\pi'(\varrho_*+\theta_0)\nabla \theta_0\|_{L^p(0,T;L^q(\Omega))}\leq c T^\frac{1}{p}\Pi L.
\end{equation}
By the assumed smoothness of $\pi$, the Sobolev embedding and \eqref{ThetaEst} we get
\begin{equation}\label{PrDiffEst4}
\begin{split}
\|(\pi'(\varrho_*+\theta_0+\vartheta)-\pi'(\varrho_*+\theta_0))\nabla(\theta_0+ \vartheta)\|_{L^p(0,T;L^q(\Omega))}&\leq \Pi\|\vartheta\|_{L^\infty(0,T;L^\infty(\Omega))}\|\nabla(\theta_0+\vartheta)\|_{L^p(0,T;L^q(\Omega))}\\
&\leq c\Pi T^\frac{1}{p'}M(L+M).
\end{split}
\end{equation}
It remains to estimate the norm of the difference $\dvr \mathcal{S}(\SymGDev w)-\mathcal{A}(\mathbb D u_0)(\mathbb Dw)$. To this end, recalling \eqref{eq:DivS} and \eqref{eq:def.A} we obtain
\begin{equation}\label{LinearizDiff}
\begin{split}
\|&\dvr \mathcal{S}(\SymGDev w)-\mathcal{A}(\mathbb D u_0)(\mathbb D(w-u_0))\|_{L^p(0,T;L^q(\Omega))}\\
\leq&\|\sum_{j,k,l,m=1}^d(a_{jk}^{lm}(\mathbb D w)-a_{jk}^{lm}(\mathbb D u_0))\partial_l\partial_m (w-u_0)\|_{L^p(0,T;L^q(\Omega))}\\
&+\|\sum_{j,k,l,m=1}^da_{jk}^{lm}(\mathbb D w)\partial_l\partial_mu_0\|_{L^p(0,T;L^q(\Omega))}.
\end{split}
\end{equation}
By \eqref{eq:ACoef} and \eqref{ACoefDifEst} it follows that
\begin{align*}
&\|a_{jk}^{lm}(\mathbb D w)-a_{jk}^{lm}(\mathbb D u_0)\|_{L^\infty(0,T;L^\infty(\Omega))}
\\&\leq cK(1+\|w\|_{\mathcal V^{p,q}(Q_T)}+\|u_0\|_{W^{2,q}(\Omega)}+\|w\|^3_{\mathcal V^{p,q}(Q_T)}+\|u_0\|^3_{W^{2,q}(\Omega)})\|w-u_0\|_{\mathcal V^{p,q}(Q_T)}.
\end{align*}
We use the above estimate and \eqref{EmbIneq} in \eqref{LinearizDiff} to deduce
\begin{equation}\label{LastDiffEst}
\begin{split}
\|\dvr \mathcal{S}(\SymGDev w)-&\mathcal{A}(\mathbb Du_0)(\mathbb D(w-u_0))\|_{L^p(0,T;L^q(\Omega))}\leq c K(1+L+M+(M+L)^3)M^2\\
&+cK(1+(M+L)+(M+L)^2)LT^\frac{1}{p}.
\end{split}
\end{equation}
We summarize \eqref{RhsGEst}, \eqref{RhsGEst2}, \eqref{TrSEEst}, \eqref{I1Est}, \eqref{I2Est}, \eqref{I3Est}, \eqref{I4Est}, \eqref{TDerEst},\eqref{PrDiffEst2}, \eqref{PrDiffEst3}, \eqref{PrDiffEst4} and \eqref{LastDiffEst} to conclude
\begin{multline*}
\|\mathscr G(\vartheta,w)\|_{L^p(0,T;W^{1,q}(\Omega))}+\|\mathscr F(\vartheta,w)\|_{L^p(0,T;L^q(\Omega))}\\ 
\leq c\left(\sum_{j=1}^{N_1} K^{k_j}L^{l_j}M^{m_j}\Pi^{|k_j-1|}T^{\alpha_j}+\sum_{j=1}^{N_2} K^{n_j}L^{o_j}M^{p_j} + T\right),
\end{multline*}
where $\alpha_j >0$, $l_j,m_j, o_j, \in\eN\cup\{0\}$, $p_j\in \eN\setminus\{1\}$, $k_j,n_j\in\{0,1\}$.

The $M$ appears on the right hand side either to some nonnegative power in a product with a nonnegative power of $T$ or to the second or higher power not multiplied by a nonnegative power of $T$. As a result, for every $c>0$ we can choose $M$ and consequently also $T$ in such way that 
$$\|\mathscr G(\vartheta,w)\|_{L^p(0,T;W^{1,q}(\Omega))}+\|\mathscr F(\vartheta,w)\|_{L^p(0,T;L^q(\Omega))}\leq c M.
$$
Theorem $\ref{thm:regularita}$ then yields $\Phi:H_{T,M}\mapsto H_{T,M}$ for these appropriately chosen $T$ and $M$. \bigskip

In order to verify that $\Phi$ is a contraction on $H_{T,M}$ we fix arbitrary pairs $(\vartheta^1,w^1), (\vartheta^2,w^2)\in H_{T,M}$. Next we note that the difference $(\theta^1-\theta^2,\tilde u^1-\tilde u^2)$ of solutions corresponding to $(\vartheta^1,w^1)$, $(\vartheta^2,w^2)$ respectively, solves the equations
\begin{equation}\label{eq:LinearizedDiff}
\begin{split}
\tder (\theta^1-\theta^2)+(\varrho_*+\theta_0)\dvr(\tilde u^1-\tilde u^2)=&\mathscr{G}(\vartheta^1,w^1)-\mathscr{G}(\vartheta^2,w^2),\\
(\varrho_*+\theta_0)\tder(w^1-w^2)-\mathcal{A}(\mathbb D u_0)(\mathbb D(\tilde u^1-\tilde u^2))+\pi'(\varrho_0+\theta_0)\nabla(\theta^1-\theta^2)=&\mathscr{F}(\vartheta^1,w^1)-\mathscr{F}(\vartheta^2,w^2).
\end{split}
\end{equation}
We begin with the estimate of the difference on the right hand side of \eqref{eq:LinearizedDiff}$_1$. Employing \eqref{ThetaEst}, the Sobolev embedding and \eqref{EmbIneq} it follows that
\begin{equation}\label{RhsGDiff1}
\begin{split}
\|\vartheta^1\dvr w^1-\vartheta^2\dvr w^2\|_{L^p(0,T;W^{1,q}(\Omega))}\leq& c\|\vartheta^1-\vartheta^2\|_{L^\infty(0,T;L^\infty(\Omega))}\|w^1\|_{L^p(0,T;W^{2,q}(\Omega))}\\&+c\|\vartheta^2\|_{L^\infty(0,T;L^\infty(\Omega))}\|w^1-w^2\|_{L^p(0,T;W^{2,q}(\Omega))}\\
\leq &cT^\frac{1}{p'}(M+L)\|\vartheta^1-\vartheta^2\|_{L^{\infty}(0,T;L^{\infty}(\Omega))}\\
&+cT^\frac{1}{p'}M\|w^1-w^2\|_{L^p(0,T;W^{2,q}(\Omega))}
\end{split}
\end{equation}
as well as
\begin{equation}\label{RhsGDiff2}
\begin{split}
&\|\nabla w^1\cdot E_{w^1}(\varrho_*+\theta_0+\vartheta^1)-\nabla w^2\cdot E_{w^2}(\varrho_*+\theta_0+\vartheta^2)\|_{L^p(0,T;W^{1,q}(\Omega))}\\
&\leq \|w^1-w^2\|_{L^p(0,T;W^{2,q}(\Omega))}\|E_{w^1}\|_{L^\infty(0,T;L^\infty(\Omega))}\|\varrho_*+\theta_0+\vartheta^1\|_{L^\infty(0,T;L^\infty(\Omega))}\\
&+\|w^2\|_{L^\infty(0,T;W^{1,\infty}(\Omega))}\|E_{w^1}-E_{w^2}\|_{L^p(0,T;W^{1,q}(\Omega))}\|\varrho_*+\theta_0+\vartheta^1\|_{L^\infty(0,T;L^\infty(\Omega))}\\
&+\|w^2\|_{L^\infty(0,T;W^{1,\infty}(\Omega))}\|E_{w^2}\|_{L^\infty(0,T;L^\infty(\Omega))}\|\vartheta^1-\vartheta^2\|_{L^p(0,T;W^{1,q}(\Omega))}\\
&\leq cMT^\frac{1}{p'}\|w^1-w^2\|_{L^p(0,T;W^{2,q}(\Omega))}(L+M)+c(M+L)^2T\|w^1-w^2\|_{L^p(0,T;W^{2,q}(\Omega))}\\
&+c(M+L)MT^\frac{1}{p'}\|\vartheta^1-\vartheta^2\|_{L^p(0,T;W^{1,q}(\Omega))}
\end{split}
\end{equation}
where we applied \eqref{ELPQEst} and
\begin{equation}\label{EDiffEst}
\|E_{w^1}-E_{w^2}\|_{L^p(0,T;W^{k,q}(\Omega))}\leq cT\|w^1-w^2\|_{L^p(0,T;W^{k+1,q}(\Omega))},\ k=0,1.
\end{equation}
following from \eqref{EExpr} and the properties of the mapping $W$.

Next we focus on te estimate of the difference on the right hand side in \eqref{eq:LinearizedDiff}$_2$. By the assumed smoothness of $\pi$ and \eqref{EDiffEst} we get
\begin{equation}\label{PrDiffEst1}
\begin{split}
\|\nabla\pi(\varrho_*+\theta_0+\vartheta^1)E_{w^1}&-\nabla\pi(\varrho_*+\theta_0+\vartheta^2)E_{w^2}\|_{L^p(0,T;L^q(\Omega))}\\
&\leq \Pi\|\vartheta^1-\vartheta^2\|_{L^p(0,T;L^q(\Omega))}\|E_{w^1}\|_{L^\infty(0,T;L^\infty(\Omega))}+\Pi\|E_{w^1}-E_{w^2}\|_{L^p(0,T;L^q(\Omega))}\\
&\leq c\Pi(T^\frac{1}{p'}(M+L)+T)[\vartheta^1-\vartheta^2,w^1-w^2]_{T,M}.
\end{split}
\end{equation}
We continue with expanding the divergence to obtain
\begin{align*}
\dvr& \left(\tilde{\mathcal S}(w^1)-\tilde{\mathcal S}(w^2)-(\mathcal S(\SymGDev w^1)-\mathcal S(\SymGDev w^2))\right)\\
=&\left(\mu(|\mathcal{G}(w^1)|^2)\dvr\mathcal{G}(w^1)-\mu(|\SymGDev w^1|^2)\dvr\SymGDev w^1\right.\\
&\left.-\left(\mu(|\mathcal{G}(w^2)|^2)\dvr\mathcal{G}(w^2)-\mu(|\SymGDev w^2|^2)\dvr\SymGDev w^2\right)\right)\\
&+2\left(\mu'(|\mathcal G(w^1)|^2)|\mathcal G(w^1)|^2\dvr\mathcal{G}(w^1)-\mu'(|\SymGDev w^1|^2)|\SymGDev w^1|^2\dvr\SymGDev w^1\right.\\
&\left.-\mu'(|\mathcal G(w^2)|^2)|\mathcal G(w^2)|^2\dvr\mathcal{G}(w^2)+\mu'(|\SymGDev w^2|^2)|\SymGDev w^2|^2\dvr\SymGDev w^2\right)\\
&+\left(\eta(\mathcal D(w^1))\nabla\mathcal{D}(w^1)-\eta(\dvr w^1)\nabla\dvr w^1\right.\\
&\left.-\eta(\mathcal D(w^2))\nabla\mathcal{D}(w^2)+\eta(\dvr w^2)\nabla\dvr w^2\right)\\
&+\left(\eta'(\mathcal{D}(w^1))\mathcal D(w^1)\nabla\mathcal D(w^1)-\eta'(\dvr w^1)\dvr w^1\nabla\dvr w^1\right.\\
&-\left.\eta'(\mathcal{D}(w^2))\mathcal D(w^2)\nabla\mathcal D(w^2)+\eta'(\dvr w^2)\dvr w^2\nabla\dvr w^2\right)=\sum_{j=1}^4I_j.
\end{align*}
We estimate the norm of each $I_j$ separately. First, denoting 
\begin{equation*}
\mathcal G^j_s=s\mathcal{G}(w^j)+(1-s)\SymGDev w^j\ j=1,2
\end{equation*}
we obtain by \eqref{MuDiff}, \eqref{MuDivDiff} and \eqref{MuPrDivDiff} 
\begin{align*}
I_1=&-\int_0^1(2\mu'(|\mathcal{G}^1_s|^2)-2\mu'(|\mathcal{G}^2_s|^2))\mathcal{G}^1_s\cdot\sym(\nabla w^1E_{w_1})^D\dvr\mathcal G^1_s\\
&+2\mu'(|\mathcal G^2_s|^2)(\mathcal{G}^1_s-\mathcal G^2_s)\cdot\sym(\nabla w^1E_{w_1})^D\dvr\mathcal G^1_s\\
&+2\mu'(|\mathcal G^2_s|^2)\mathcal{G}^2_s\cdot\sym(\nabla(w^1-w^2)E_{w^1})^D\dvr\mathcal G^1_s\\
&+2\mu'(|\mathcal G^2_s|^2)\mathcal{G}^2_s\cdot\sym(\nabla w^2(E_{w^1}-E_{w^2}))^D\dvr\mathcal G^1_s\\
&+2\mu'(|\mathcal G^2_s|^2)\mathcal{G}^2_s\cdot\sym(\nabla w^2E_{w^2})^D\dvr(\mathcal G^1_s-\mathcal{G}^2_s)\\
&+(\mu(|\mathcal{G}^1_s|^2)-\mu(|\mathcal{G}^2_s|^2))\dvr\sym(\nabla w^1E_{w^1})^D\\
&+\mu(|\mathcal{G}^2_s|^2)\dvr\sym(\nabla (w^1-w^2)E_{w^1})^D\\
&+\mu(|\mathcal{G}^2_s|^2)\dvr\sym(\nabla w^2(E_{w^1}-E_{w^2}))^D\ds
\end{align*} 
and
\begin{align*}
I_2=&-\int_0^1 2\left((\mu''(|\mathcal G^1_s|^2)-\mu''(|\mathcal G^2_s|^2))|\mathcal G^1_s|^2+\mu'(|\mathcal G^1_s|^2)-\mu'(|\mathcal G^2_s|^2)\right)\\
&\mathcal G^1_s\cdot\sym(\nabla w^1E_{w^1})^D\dvr\mathcal G^1_s+2\mu''(|\mathcal G^2_s|^2)(|\mathcal G^1_s|^2-|\mathcal G^2_s|^2)\mathcal G^1_s\cdot\sym(\nabla w^1E_{w^1})^D\dvr\mathcal G^1_s\\
&+2\left(\mu''(|\mathcal G^2_s|^2)|\mathcal G^2_s|^2+\mu'(|\mathcal G^2_s|^2)\right)\left(\mathcal G^1_s-\mathcal G^2_s\right)\cdot\sym(\nabla w^1E_{w^1})^D\dvr\mathcal G^1_s\\
&+2\left(\mu''(|\mathcal G^2_s|^2)|\mathcal G^2_s|^2+\mu'(|\mathcal G^2_s|^2)\right)\mathcal G^2_s\cdot\sym(\nabla(w^1-w^2)E_{w^1})^D\dvr\mathcal G^1_s\\
&+2\left(\mu''(|\mathcal G^2_s|^2)|\mathcal G^2_s|^2+\mu'(|\mathcal G^2_s|^2)\right)\mathcal G^2_s\cdot\sym(\nabla w^2(E_{w^1}-E_{w^2}))^D\dvr\mathcal G^1_s\\
&+2\left(\mu''(|\mathcal G^2_s|^2)|\mathcal G^2_s|^2+\mu'(|\mathcal G^2_s|^2)\right)\mathcal G^2_s\cdot\sym(\nabla w^2E_{w^2})^D\dvr(\mathcal G^1_s-\mathcal G^2_s)\\
&+(\mu(|\mathcal G^1_s|^2)-\mu(|\mathcal G^2_s|^2))|\mathcal G^1_s|^2\dvr\sym(\nabla w^1E_{w^1})^D\\
&+\mu(|\mathcal G^2_s|^2)(|\mathcal G^1_s|^2-|\mathcal G^2_s|^2)\dvr\sym(\nabla w^1E_{w^1})^D\\
&+\mu(|\mathcal G^2_s|^2)|\mathcal G^2_s|^2\dvr\sym(\nabla (w^1-w^2)E_{w^1})^D\\
&+\mu(|\mathcal G^2_s|^2)|\mathcal G^2_s|^2\dvr\sym(\nabla w^2(E_{w^1}-E_{w^2}))^D\ds.
\end{align*}
Then denoting
\begin{equation*}
\mathcal D^j_s=s\mathcal D(w^j)+(1-s)\dvr w^j\ j=1,2
\end{equation*}
we obtain using \eqref{LambdaNabDiff1}
\begin{align*}
I_3=&-\int_0^1 \left(\eta'(\mathcal D^1_s)-\eta'(\mathcal D^2_s)\right)\nabla w^1\cdot E_{w^1}\nabla\mathcal D^1_s+\eta'(\mathcal D^2_s)\nabla (w^1-w^2)\cdot E_{w^1}\nabla\mathcal D^1_s\\
&+\eta'(\mathcal D^2_s)\nabla w^2\cdot (E_{w^1}-E_{w^2})\nabla\mathcal D^1_s+\eta'(\mathcal D^2_s)\nabla w^2\cdot E_{w^2}\nabla(\mathcal D^1_s-\mathcal D^2_s)\\
&+(\eta(\mathcal D^1_s)-\eta(\mathcal D^2_s))\nabla(\nabla w^1\cdot E_{w^1})+\eta(\mathcal D_s^2)\nabla(\nabla(w^1-w^2)\cdot E_{w^1})
\\
&+\eta(\mathcal D_s^2)\nabla(\nabla w^2\cdot (E_{w^1}-E_{w^2}))\ds
\end{align*}
and
\begin{align*}
I_4=&-\int_0^1 \left(\eta''(\mathcal D^1_s)-\eta''(\mathcal D^2_s)\right)\nabla w^1\cdot E_{w^1}\mathcal D^1_s\nabla\mathcal D^1_s+\eta''(\mathcal D^2_s)\nabla(w^1-w^2)\cdot E_{w^1}\mathcal D^1_s\nabla\mathcal D^1_s\\
&+\eta''(\mathcal D^2_s)\nabla w^2\cdot (E_{w^1}-E_{w^2})\mathcal D^1_s\nabla\mathcal D^1_s+\eta''(\mathcal D^2_s)\nabla w^2\cdot E_{w^2}(\mathcal D^1_s-\mathcal D^2_s)\nabla\mathcal D^1_s\\
&+\eta''(\mathcal D^2_s)\nabla w^2\cdot E_{w^2}\mathcal D^2_s\nabla(\mathcal D^1_s-\mathcal D^2_s)+\left(\eta'(\mathcal D^1_s)-\eta'(\mathcal D^2_s)\right)\mathcal D^1_s\nabla(\nabla w^1\cdot E_{w^1})\\
&+\eta'(\mathcal D^2_s)\left(\mathcal D^1_s-\mathcal D^2_s\right)\nabla(\nabla w^1\cdot E_{w^1})+\eta'(\mathcal D^2_s)\mathcal D^2_s\nabla(\nabla(w^1-w^2)\cdot E_{w^1})\\
&+\eta'(\mathcal D^2_s)\mathcal D^2_s\nabla(\nabla w^2\cdot (E_{w^1}-E_{w^2}))\ds.
\end{align*}

We also have the estimates
\begin{align*}
\|\mathcal G^1_s&-\mathcal G^2_s\|_{L^\infty(0,T;L^\infty(\Omega))}\\
\leq&
\|\SymGDev(w^1-w^2)\|_{L^\infty(0,T;L^\infty(\Omega))}+\|\nabla(w^1-w^2)\|_{L^\infty(0,T;L^\infty(\Omega))}\|E_{w^1}\|_{L^\infty(0,T;L^\infty(\Omega))}\\
&+\|\nabla w^2\|_{L^\infty(0,T;L^\infty(\Omega))}\|(E_{w^1}-E_{w^2})\|_{L^\infty(0,T;L^\infty(\Omega))}\\
\leq&c\|w^1-w^2\|_{\mathcal V^{p,q}(Q_T)}(1+T^\frac{1}{p'}(\|w^1\|_{L^p(0,T;W^{2,q}(\Omega))}+\|w^2\|_{\mathcal V^{p,q}(Q_T)}))
\end{align*}
and
\begin{align*}
\|\dvr&(\mathcal G^1_s-\mathcal G^2_s)\|_{L^p(0,T;L^q(\Omega))}\\
\leq& c\|\nabla^2w^1-w^2\|_{L^p(0,T;L^q(\Omega))}(1+\|E_{w^2}\|_{L^\infty(0,T;L^\infty(\Omega))})\\
&+c\|\nabla(w^1-w^2)\|_{L^\infty(0,T;L^\infty(\Omega))}\|\nabla E_{w^2}\|_{L^p(0,T;L^q(\Omega))}\\
&+\|\nabla^2w^1\|_{L^p(0,T;L^q(\Omega))}\|E_{w^2-w^1}\|_{L^\infty(0,T;L^\infty(\Omega))}\\
&+\|\nabla w^1\|_{L^\infty(0,T;L^\infty(\Omega))}\|\nabla E_{w^2-w^1}\|_{L^p(0,T;L^q(\Omega))}\\
\leq&c\|w^1-w^2\|_{\mathcal V^{p,q}(Q_T)}(1+(T^\frac{1}{p'}+T)(\|w^1\|_{\mathcal V^{p,q}(Q_T)}+\|w^2\|_{\mathcal{V}^{p,q}(Q_T)}))
\end{align*}
by \eqref{EEst} and \eqref{EmbIneq}. By the previous two estimates along with \eqref{GSEst}, \eqref{DivGSEst}, \eqref{DivSymEst}, \eqref{DSEst} and \eqref{NabDSEst} in straightforward manipulations similar to the ones in the boundedness proof we infer that
\begin{equation}\label{AuxSum}
\sum_{j=1}^4\|I_j\|_{L^p(0,T;L^q(\Omega))}\leq \|w^1-w^2\|_{\mathcal V^{p,q}(Q_T)}cK\sum_{j=1}^{\tilde N} M^{p_j}L^{r_j}T^{\kappa_j},\ p_j,r_j\in\eN\cup \{0\},\kappa_j>0.
\end{equation}
We omit the detailed presentation of the proof of the latter inequality in order to keep the length of the paper in a reasonable limit. By very similar manipulations we obtain for the last term of the difference in \eqref{eq:LinearizedDiff}$_2$ that
\begin{equation}\label{TrSEDiffEst}
\begin{split}
\|&\tr(\nabla \tilde{\mathcal S}(w^1)E_{w^1})-\tr(\nabla\tilde{\mathcal S}(w^2)E_{w^2})\|_{L^p(0,T;L^q(\Omega))}\\
&\leq \|w^1-w^2\|_{\mathcal V^{p,q}(Q_T)}cK\sum_{j=1}^{\tilde P}M^{p_j}L^{r_j}T^{\kappa_j},\ p_j,r_j\in\eN\cup \{0\},\kappa_j>0.
\end{split}
\end{equation}

We infer in a straightforward way
\begin{equation}\label{TDerDiffEst}
\begin{split}
\|\vartheta^1\tder w^1-\vartheta^2 \tder w^2\|_{L^p(0,T;L^q(\Omega))}\leq& \|\vartheta^1-\vartheta^2\|_{L^\infty(0,T;L^\infty(\Omega))}\|\tder w^1\|_{L^p(0,T;L^q(\Omega))}\\&+\|\vartheta^2\|_{L^\infty(0,T;L^\infty(\Omega))}\|\tder w^1-\tder w^2\|_{L^p(0,T;L^q(\Omega))}\\
\leq &cT^\frac{1}{p'}M\left(\|\vartheta^1-\vartheta^2\|_{W^{1,p}(0,T;W^{1,q}(\Omega))}+\|\tder (w^1-w^2)\|_{L^p(0,T;L^q(\Omega))}\right)
\end{split}
\end{equation}
and
\begin{equation}\label{PrDiff3Est}
\begin{split}
&\|(\pi'(\varrho_*+\theta_0+\vartheta^1)-\pi'(\varrho_*+\theta_0))\nabla(\theta_0+\vartheta^1)-(\pi'(\varrho_*+\theta_0+\vartheta^2)-\pi'(\varrho_*+\theta_0))\nabla(\theta_0+\vartheta^2)\|_{L^p(0,T;L^q(\Omega))}\\
&\leq \Pi\|\vartheta^1-\vartheta^2\|_{L^\infty(0,T;L^\infty(\Omega))}\|\theta_0+\vartheta^1\|_{L^p(0,T;W^{1,q}(\Omega))}+2\Pi\|\vartheta^1-\vartheta^2\|_{L^p(0,T;W^{1,q}(\Omega))}
\\
&\leq c\Pi T^\frac{1}{p'}(L+M+1)\|\vartheta^1-\vartheta^2\|_{W^{1,p}(0,T;W^{1,q}(\Omega))}.
\end{split}
\end{equation}

Applying \eqref{ACoefDifEst} and \eqref{EmbIneq} we have
\begin{equation}\label{LTermDiffEst}
\begin{split}
\|&\dvr(\mathcal S(\SymGDev w^1)-\mathcal S(\SymGDev w^2))-\mathcal A(\mathbb Du_0)(\mathbb D(w^1-w^2))\|_{L^p(0,T;L^q(\Omega))}\\
=&\|\mathcal{A}(\mathbb Dw^1)(\mathbb Dw^1)-\mathcal A(\mathbb Dw^2)(\mathbb Dw^2)-\mathcal A(\mathbb Du_0)(\mathbb D(w^1-w^2))\|_{L^p(0,T;L^q(\Omega))}\\
\leq& \|\mathcal{A}(\mathbb Dw^1)(\mathbb D(w^1-w^2))-\mathcal A(\mathbb Du_0)(\mathbb D(w^1-w^2))\|_{L^p(0,T;L^q(\Omega))}\\
&+\|\mathcal{A}(\mathbb Dw^1)(\mathbb Dw^2)-\mathcal A(\mathbb Dw^2)(\mathbb Dw^2) \|_{L^p(0,T;L^q(\Omega))}\\
\leq&\sum_{j,k,l,m=1}^d\left(\|a_{jk}^{lm}(\mathbb Dw^1)-a_{jk}^{lm}(\mathbb Du_0)\|_{L^\infty(0,T;L^\infty(\Omega))}\|\partial_k\partial_l(w^1-w^2)\|_{L^p(0,T;L^q(\Omega))}\right.\\
&\left.+\|a_{jk}^{lm}(\mathbb Dw^1)-a_{jk}^{lm}(\mathbb Dw^2)\|_{L^\infty(0,T;L^\infty(\Omega))}\|\partial_k\partial_l w^2\|_{L^p(0,T;L^q(\Omega))}\right)\\
\leq &cK(1+\|w^1\|_{L^\infty(0,T;W^{1,\infty}(\Omega))}+\|u_0\|_{W^{1,\infty}(\Omega)}+\|w^1\|^3_{L^\infty(0,T;W^{1,\infty}(\Omega))}+\|u_0\|^3_{W^{1,\infty}(\Omega)})\\
&\times\|w^1-u_0\|_{L^\infty(0,T;W^{1,\infty}(\Omega))}\|w^1-w^2\|_{L^p(0,T;W^{2,q}(\Omega))}\\
&+cK(1+\|w^1\|_{L^\infty(0,T;W^{1,\infty}(\Omega))}+\|w^2\|_{L^\infty(0,T;W^{1,\infty}(\Omega))}+\|w^1\|^3_{L^\infty(0,T;W^{1,\infty}(\Omega))}+\|w^2\|^3_{L^\infty(0,T;W^{1,\infty}(\Omega))})\\
&\times\|w^1-w^2\|_{L^\infty(0,T;W^{1,\infty}(\Omega))}\|w^2\|_{L^p(0,T;W^{2,q}(\Omega))}\\
\leq& cK(1+(M+L)+L+(M+L)^3+L^3)M\|w^1-w^2\|_{\mathcal V^{p,q}(Q_T)}\\
&+cK(M+L)(1+(M+L)+(M+L)^3)\|w^1-w^2\|_{\mathcal V^{p,q}(Q_T)}.
\end{split}
\end{equation}

Taking into account \eqref{RhsGDiff1}, \eqref{RhsGDiff2}, \eqref{PrDiffEst1}, \eqref{TrSEDiffEst}, \eqref{AuxSum}, \eqref{TDerDiffEst}, 
 \eqref{PrDiff3Est} and \eqref{LTermDiffEst} we conclude
\begin{equation*}
\begin{split}
&\|\mathscr G(\vartheta^1,w^1)-\mathscr G(\vartheta^2,w^2)\|_{L^p(0,T;W^{1,q}(\Omega))}+\|\mathscr F(\vartheta^1,w^1)-\mathscr F(\vartheta^2,w^2)\|_{L^p(0,T;L^q(\Omega))}\\
&\leq [\vartheta^1-\vartheta^2,w^1-w^2]_{T,M}c\left(\sum_{j=1}^{Q_1} K^{k_j}L^{l_j}M^{m_j}\Pi^{|k_j-1|}T^{\beta_j}+\sum_{j=1}^{Q_1} K^{n_j}L^{o_j}M^{p_j}\right),
\end{split}
\end{equation*}
where $l_j,m_j,o_j\in\eN\cup\{0\}$, $\beta_j>0$, $p_j\in\eN\setminus\{1\}$ and $k_j,n_j\in\{0,1\}$. 

As before, $T$ and $M$ can be chosen such that the sum in the bracket on the right hand side is less than $c$ for arbitrary $c>0$. Consequently, $T$ and $M$ can be chosen such that $\Phi$ is a contraction according to Theorem \ref{thm:regularita} and the Banach fix-point theorem then yields the validity of Theorem \ref{thm.main}. 

It only remains to prove Theorem \ref{thm:regularita}. This is content of the next section.



\section{\texorpdfstring{$L^p$}{someref} regularity of the linearized system}
\label{Lpreg}
The main aim of this  section is to prove the following regularity theorem for \eqref{eq:linearized}. For simplicity, we denote $u = \tilde u - u_0$ throughout this section. 

\begin{Theorem}\label{thm:regularita} Let $\theta_0,\ \nabla u_0\in BUC(\Omega)$\footnote{Here $BUC$ denotes the space of bounded and uniformly continuous functions. Note that $W^{1,q}(\Omega)\hookrightarrow BUC(\Omega)$ whenever $q>d$.} and let $\varrho_*+\theta_0\in L^\infty(\Omega)$ and $(\varrho_* + \theta_0)^{-1}\in L^\infty(\Omega)$. Then for every $(\mathscr G, \mathscr F)\in L^p(0,T;W^{1,q}(\Omega))\times L^p(0,T;L^q(\Omega))$ the solution $(\theta, u)$ to \eqref{eq:linearized} satisfies
	\begin{equation*}
	\|  u\|_{\mathcal V^{p,q}(Q_T)} + \|\theta\|_{W^{1,p}(0,T;W^{1,q}(\Omega))}
	\leq c\left(\|\mathscr G\|_{L^p(0,T;W^{1,q}(\Omega))} + \|\mathscr F\|_{L^p(0,T;L^q(\Omega))}\right)
	\end{equation*}
	where $c$ is a constant independent of $\theta,\  u$ and the right hand side.
\end{Theorem}

In order to reach this goal we use a semi-group approach which may be found in e.g. \cite{DeHiPr} or in \cite{EnSh}. In particular, the equation \eqref{eq:linearized} can be rewritten as
\begin{equation}\label{eq:A.operator}
\partial_t \left(\begin{matrix} \theta\\  u\end{matrix}\right)  - \mathscr A \left(\begin{matrix}\theta \\ u \end{matrix}\right)  = \left(\begin{matrix} \mathscr G\\ \mathscr F\end{matrix}\right)
\end{equation}
for $\mathscr A$ defined as
$$
\mathscr A\left(\begin{matrix}\theta\\  u\end{matrix}\right) = \left(\begin{matrix} - (\varrho_* + \theta_0)\dvr  u \\ \frac 1{\varrho_*+\theta_0}\mathcal A(\mathbb D u_0)(\mathbb D u) - \frac 1 {\varrho_* + \theta_0}  \pi'(\varrho_* + \theta_0)\nabla \theta\end{matrix}.\right)$$
The equation \eqref{eq:A.operator} is equipped with zero initial conditions, i.e. $\theta(0,\cdot) \equiv 0$ and $u(0,\cdot) \equiv 0$. 
Note that here we write $\mathscr F$ instead of $\frac 1{\varrho_*+\theta_0}\mathscr F$. We can afford this inaccuracy since $\frac 1{\varrho_* + \theta_0}$ and $\varrho_* + \theta_0$  are bounded in $L^\infty$ in all applications and thus
\begin{equation*}
\frac1{\varrho_* + \theta_0}\mathscr F \in L^p(0,T;L^q(\Omega)) \Leftrightarrow \mathscr F \in L^p(0,T;L^q(\Omega)).
\end{equation*}

The resolvent problem for such equation is
\begin{equation}\label{eq:lin.operator}
\lambda\left(\begin{matrix} \theta\\  u\end{matrix}\right) - \mathscr A \left(\begin{matrix}\theta\\  u\end{matrix}\right) = \left(\begin{matrix}\mathscr G\\\mathscr F\end{matrix}\right).
\end{equation}
First, we recall the definition of $\mathcal R-$boundedness (see \cite[Definition 2.1]{EnSh}):

\begin{Definition}
	A family $\mathcal T\subset \mathcal L(X,Y)$ (i.e. a family of linear operators from $X$ to $Y$) is called $\mathcal R-$bounded if there exist constants $C>0$ and $p\in [1,\infty)$ such that for each $n\in \mathbb N$, $T_j\in \mathcal T$, $f_j\in X$ ($j\in\{1,\ldots, n\}$) and for all sequences $\{r_j(v)\}_{j=1}^n$ of independent, symmetric, $\{-1,1\}$-valued random variables on $[0,1]$ there holds 
	$$
	\int_0^1 \left\|\sum_{j=1}^n r_j(v) T_j f_j\right\|_Y^p\ {\rm d}v \leq C\int_0^1 \left\|\sum_{j=1}^n r_j(v) f_j \right\|_X^p\ {\rm d}v.
	$$
	The smallest such $C$ is called $\mathcal R-$bound and it is denoted by $\mathcal R_{\mathcal L(X,Y)} (\mathcal T)$.
\end{Definition}
We aim to prove the $\mathcal R$--sectoriality of the operator $\mathscr A$, i.e., to justify that $\mathscr R_\lambda = \lambda(\lambda \mathbb I_d - \mathscr A)^{-1}$  is $\mathcal R-$bounded for all $\lambda\in \Sigma_{\beta,\nu}$ for some $\beta\in (0,\frac \pi 2)$ and $\nu>0$. Here
$$
\Sigma_{\beta,\nu} = \{z\in \mathbb C,\ |{\rm arg}(z-\nu)|\leq \pi-\beta\}.
$$

In the last part of this section we show that this $\mathcal R-$boundedness is sufficient for the $L^p-L^q$ regularity. To reach this goal, we use the following theorem which is due to Weis.
\begin{Theorem} Let $X$ and $Y$ be two $UMD$ spaces and $p\in (1,\infty)$. Let $M$ be a function in $C^{1}(\mathbb R\setminus \{0\},\mathcal L(X,Y))$ such that\label{thm.weis}
	\begin{equation*}
	\begin{split}
	&\mathcal R_{\mathcal L(X,Y)} \left(\{M(\tau),\ \tau \in \mathbb R\setminus \{0\}\}\right) =\kappa_0<\infty\\
	& \mathcal R_{\mathcal L(X,Y)}\left(\{\tau M'(\tau),\ \tau\in \mathbb R\setminus\{0\} \}\right) = \kappa_1<\infty.
	\end{split}
	\end{equation*}
	Then $T_M$ defined as $T_M\Phi = \mathcal F^{-1}[M\mathcal F(\Phi)]$ is extended to a bounded linear operator from $L^p(\mathbb R,X)$ into $L^p(\mathbb R,Y)$. Moreover,
	\begin{equation*}
	\|T_M\|_{\mathcal L(X,Y)} \leq c (\kappa_0  + \kappa_1)
	\end{equation*}
	where $c>0$ is independent of $p,\ X,\ Y,\ \kappa_0$ and $\kappa_1$.
\end{Theorem}

Here $UMD$ stands for spaces with uniform martingale differences. Their properties might be found in \cite{burkholder} or in \cite{francia}. For our purposes it is enough to know that $L^p$ and $W^{k,p}$ are $UMD$ once $p\in (1,\infty).$

Since we assume $\lambda\neq 0$ in \eqref{eq:lin.operator}, we deduce
\begin{equation*}
\theta = \frac 1\lambda \left(\mathscr G - (\varrho_* + \theta_0) \dvr u\right).
\end{equation*}
Consequently, we may rewrite the resolvent problem as
\begin{equation}\label{eq:hvezdicka}
\lambda u - \mathcal B_\lambda (u) = F_\lambda\qquad \mbox{ on }\Omega
\end{equation}
where
\begin{equation*}
\begin{split}
&\mathcal B_\lambda(u) = \frac 1{\varrho_*+\theta_0} \mathcal A(\mathbb D u_0)(\mathbb D u) +  \pi'(\varrho_* + \theta_0) \frac 1\lambda \nabla\dvr u\\
&F_\lambda  = T_\lambda\left(\begin{matrix}\mathscr G\\ \mathscr F\end{matrix}\right) := \left( \mathscr F - \nabla \mathscr G \frac 1\lambda \right).
\end{split}
\end{equation*}

We claim the following.
\begin{Theorem}
	\label{thm.non-constant} Let $\mathcal U_\lambda:L^q(\Omega)\to W^{2,q}(\Omega)$ be a solution operator to \eqref{eq:hvezdicka}. Then there exists a positive constant $c$ such that\footnote{Here $\mathcal R_{\mathcal L(L^q(\Omega))}$ is an abbreviation of $\mathcal R_{\mathcal L(L^q(\Omega), L^q(\Omega))}$}
	\begin{equation*}
	\begin{split}
	\mathcal R_{\mathcal L(L^q(\Omega))} \left(\left\{\lambda \mathcal U_\lambda, \lambda \in \Sigma_{\beta,\nu}\right\}\right) &\leq c\\
	\mathcal R_{\mathcal L(L^q(\Omega))} \left(\left\{|\lambda|^{1/2} \partial_j\mathcal U_\lambda, \lambda \in \Sigma_{\beta,\nu}\right\}\right) &\leq c\\
	\mathcal R_{\mathcal L(L^q(\Omega))} \left(\left\{\partial_j\partial_k \mathcal U_\lambda, \lambda \in \Sigma_{\beta,\nu}\right\}\right) &\leq c
	\end{split}
	\end{equation*}
	for every $j, k \in \{1,\ldots,d\}$, every $q\in (1,\infty)$ and for some $\beta\in (0,\pi/2)$ and $\nu>0$. 
\end{Theorem}

Note that the $\mathcal R-$boundedness of $\mathscr R_\lambda$ is a consequence of this Theorem and Proposition \ref{pro.2.16.ensh}. Indeed, recall that $\mathscr R_\lambda$ has two components. In particular, $\mathscr R_{\lambda 1} \left(\begin{matrix}\mathscr G\\ \mathscr F\end{matrix}\right) = \lambda\theta$ and $\mathscr R_{\lambda 2} \left(\begin{matrix} \mathscr G\\ \mathscr F\end{matrix}\right) = \lambda u$. Both of them can be written as a composition and sum of certain operators, namely
\begin{equation*}
\begin{split}
\mathscr R_{\lambda 1} \left(\begin{matrix} \mathscr G\\ \mathscr F\end{matrix}\right) &= \left( -(\varrho_* + \theta_0) \dvr \mathcal U_\lambda\circ T_\lambda\left(\begin{matrix}\mathscr G\\\mathscr F\end{matrix}\right) + \mathscr G\right)\\
\mathscr R_{\lambda 2} \left(\begin{matrix}\mathscr G\\ \mathscr F\end{matrix}\right) & =  \left(\lambda \mathcal U_\lambda \circ T_\lambda \right) \left(\begin{matrix} \mathscr G\\ \mathscr F\end{matrix}\right).
\end{split}
\end{equation*}
Proposition \ref{pro.2.16.ensh} then yields the claim.

In order to show Theorem \ref{thm.non-constant} we first prove its version with constant coefficients.

\begin{Theorem}\label{Thm:ConstCase}
	Let $\mathscr B_\lambda(u):W^{2,q}(\mathbb R^d)\to L^{q}(\mathbb R^d)$ be an operator defined as
	\begin{equation*}
	\mathscr B_\lambda(u) = \frac 1{\gamma_1} \sum_{k,l,m = 1}^da_{jk}^{lm} (|D|) \partial_l\partial_m u_k + \frac{\gamma_2}{\lambda}\sum_{l=1}^d\partial_j\partial_l u_l
	\end{equation*}
	where $a_{jk}^{lm}$ is defined by $\eqref{eq:ACoef}$ and $D\in \mathbb R^{d\times d}_{{\rm sym}}$, $\gamma_1,\ \gamma_2\in \mathbb R^+$ are arbitrary. Then the solution operator $\mathscr U_\lambda: L^q(\mathbb R^d)\to W^{2,q}(\mathbb R^d)$, $\mathscr U_\lambda: f\mapsto u$ where $u$ solves
	\begin{equation}\label{eq:const.coef}
	\lambda u - \mathscr B_\lambda u = f
	\end{equation}
	in $\mathbb R^d$ fulfills
	\begin{equation*}
	\begin{split}
	\mathcal R_{L^q(\mathbb R^d)} \left(\left\{\lambda \mathscr U_\lambda, \lambda \in \Sigma_{\beta,\nu}\right\}\right) &\leq c\\
	\mathcal R_{L^q(\mathbb R^d)} \left(\left\{|\lambda|^{1/2} \partial_j\mathscr U_\lambda, \lambda \in \Sigma_{\beta,\nu}\right\}\right) &\leq c\\
	\mathcal R_{L^q(\mathbb R^d)} \left(\left\{\partial_j\partial_k \mathscr U_\lambda, \lambda \in \Sigma_{\beta,\nu}\right\}\right) &\leq c
	\end{split}
	\end{equation*}
	for every $j, k \in \{1,\ldots,d\}$, every $q\in (1,\infty)$ and for some $\beta\in (0,\pi/2)$ and $\nu>0$. 
\end{Theorem}

\begin{proof}
	We apply Fourier transform to rewrite \eqref{eq:const.coef} as
	\begin{equation*}
	\lambda \hat u_m  + \frac 1{\gamma_1} \sum_{j,k,l=1}^da_{mj}^{kl} (|D|) \xi_k\xi_l \hat u_j  + \frac{\gamma_2}{\lambda} \xi_m\sum_{l=1}^d\xi_l \hat u_l = \hat f_m
	\end{equation*}
	where
	\begin{equation*}
	\hat u(\xi) := \mathcal F(u) = \int_{\mathbb R^d} e^{-i\eta\xi} u(\eta) \ {\rm d}\eta.
	\end{equation*}
	Set
	\begin{equation*}
	\begin{split}
	E(\xi)_{mj} &= \sum_{k,l=1}^d\frac 1{\gamma_1} a_{mj}^{kl} (|D|) \xi_k\xi_l\\
	E(\xi, \lambda)_{ml} &= \frac{\gamma_2}{\lambda} \xi_m\xi_l.
	\end{split}
	\end{equation*}
	Naturally,
	\begin{equation*}
	u = \mathcal F^{-1} \left((\lambda \mathbb I_d + E(\xi) + E(\lambda,\xi))^{-1}\hat f\right)
	\end{equation*}
	assuming $(\lambda \mathbb I_d + E(\xi) + E(\lambda,\xi))^{-1}$ exists -- this is discussed below. Further,
	\begin{equation}\label{eq:symboly}
	\begin{split}
	\lambda u &= \mathcal F^{-1}\left(\lambda (\lambda \mathbb I_d + E(\xi) + E(\lambda,\xi))^{-1}\hat f\right)\\
	|\lambda|^{\frac 12} \partial_j u & = \mathcal F^{-1} \left(|\lambda|^{\frac 12} i \xi_j(\lambda \mathbb I_d + E(\xi) + E(\lambda,\xi))^{-1}\hat f\right)\\
	\partial_j\partial_k u & = \mathcal F^{-1} \left( - \xi_j\xi_k(\lambda \mathbb I_d + E(\xi) + E(\lambda,\xi))^{-1}\hat f\right).
	\end{split}
	\end{equation}
	According to Theorem \ref{thm.3.3.ensh} it suffices to show that the multipliers appearing in \eqref{eq:symboly} satisfies
	\begin{equation}\label{eq:multipl}
	|\partial_\xi^\alpha m(\lambda,\xi)|\leq c_\alpha |\xi|^{-|\alpha|}
	\end{equation}
	for all multi-indeces $\alpha \in\mathbb N_0^d$ and for all $\lambda \in \Sigma_{\beta,\nu}$ where $\beta$ and $\nu$ are appropriately chosen. \\
	First, consider a matrix $\lambda \mathbb I_d + E(\xi)$. Since $E$ is elliptic and symmetric (see \eqref{StrongEll} and \eqref{eq:a.symmetry}) there is an orthogonal $\xi-$dependent matrix $O$ such that
	\begin{equation*}
	E(\xi) = O \left(\begin{matrix}\mu_1 & 0 & 0\\ 0& \mu_2 & 0\\ 0 & 0& \mu_3 \end{matrix}\right) O^T.
	\end{equation*}
	We restrict ourselves here to the case of $3\times 3$--matrices for clarity of the presentation.
	We consider without loss of generality $\mu_3\geq\mu_2\geq \mu_1 \geq c |\xi|^2$ (here $c$ in fact depend on $D$, however, we assume in applications that $D$ is always bounded by some constant dependent on $u_0$). Consequently,
	\begin{equation*}
	\lambda \mathbb I_d + E(\xi) = O \left(\begin{matrix} \lambda + \mu_1 & 0 & 0\\ 0 & \lambda + \mu_2 & 0 \\ 0 & 0 & \lambda + \mu_3\end{matrix}\right) O^T
	\end{equation*}
	and there exists $(\lambda \mathbb I_d + E(\xi))^{-1}$ which may be written as
	\begin{equation*}
	(\lambda \mathbb I_d + E(\xi))^{-1} = O \left(\begin{matrix} \frac 1{\lambda + \mu_1} & 0 & 0\\ 0 & \frac1{\lambda + \mu_2} & 0 \\ 0 & 0 & \frac{1}{\lambda + \mu_3}\end{matrix}\right) O^T.
	\end{equation*}
	We use Lemma \ref{lem.3.1.ensh} and the orthogonality of $O$ to claim
	\begin{equation*}
	|(\lambda \mathbb I_d + E(\xi))^{-1}|\leq \frac{c(\nu)}{|\lambda| + |\xi|^2}
	\end{equation*}
	in an operator norm. Further
	\begin{equation*}
	(\lambda \mathbb I_d + E(\xi) + E(\xi,\lambda))^{-1} = (\lambda \mathbb I_d + E(\xi))^{-1}(\mathbb I_d + (\lambda \mathbb I_d + E(\xi))^{-1} E(\xi,\lambda))^{-1}
	\end{equation*}
	and, from definition,
	\begin{equation*}
	|(\lambda \mathbb I_d + E(\xi))^{-1} E(\xi,\lambda)|\leq \frac{\gamma_2}{|\lambda|} \frac{c(\nu) |\xi|^2}{|\lambda| + c|\xi|^2}\leq \frac12
	\end{equation*}
	assuming $|\lambda|$ is sufficiently high. This can be managed by a proper choice of $\nu$ and $\beta$. For this choice we get $\lambda \mathbb I_d + E(\xi) + E(\xi,\lambda)$ invertible and
	\begin{equation*}
	\left|(\lambda \mathbb I_d + E(\xi) + E(\xi,\lambda))^{-1}\right|\leq \frac{c(\nu)}{|\lambda| + |\xi|^2}.
	\end{equation*}
	
	To proceed further, we prove the following lemma:
	\begin{Lemma}
		Let $|\lambda|>\lambda_0>0$. There exist coefficients $a_{j,n}$ depending on $\lambda_0$  such that
		\begin{equation}\label{eq:odhad.na.derivaci}
		\left|\partial^\alpha_\xi (\lambda \mathbb I_d + E(\xi) + E(\xi,\lambda))^{-1}\right| \leq \sum_{j=0}^{\left\lfloor \frac n2 \right\rfloor} a_{j,n} |\xi|^{n-2j} \left(\frac{c(\nu)}{|\lambda| + |\xi|^2}\right)^{1+n-j}
		\end{equation}
		for all multi-indeces $\alpha$ with $|\alpha| = n$
	\end{Lemma}
	\begin{proof}
		During this proof, we assume that all matrix multiplications appearing in the proof are commutative. This allows to simplify the notation and it does not affect the validity of the main claim. Using this assumption, we are going to prove 
		\begin{equation}\label{eq:indukce}
		\partial^\alpha_\xi(\lambda \mathbb I_d  + E(\xi) + E(\xi,\lambda))^{-1} = \sum_{j=0}^{\left\lfloor \frac n2\right\rfloor} a_{j,n} \left(E'(\xi) + E'(\xi,\lambda)\right)^{n-2j} \left(\lambda \mathbb I + E(\xi) + E(\xi,\lambda)\right)^{-1-n+j}.
		\end{equation}
		Here we also use another simplification as $E'(\xi)$ denotes arbitrary first partial derivative and the bracket $\left(E'(\xi) + E'(\xi,\lambda)\right)^{n-2j}$ should be read as the product of $n-2j$ first derivatives which are not necessarily with respect to the same coordinate. We also assume that $a_{j,n}$ is independent of $\lambda$ as far as $|\lambda|$ is far away from zero. 
		
		Clearly, \eqref{eq:odhad.na.derivaci} follows directly from \eqref{eq:indukce}. We prove \eqref{eq:indukce} by induction. Let $n=0$, then
		\begin{equation*}
		\partial^0 (\lambda \mathbb I_d + E(\xi) + E(\xi,\lambda))^{-1} = (\lambda \mathbb I_d + E(\xi) + E(\xi,\lambda))^{-1}.
		\end{equation*}
		Now let \eqref{eq:indukce} holds for certain $n\in \mathbb N$ and let $\alpha$ be a multiindex of length $|\alpha| = n+1$. Then
		\begin{equation*}
		\partial^\alpha_\xi(\lambda \mathbb I_d + E(\xi) + E(\xi,\lambda))^{-1} = \left(\partial^{\alpha'}_\xi(\lambda \mathbb I_d + E(\xi) + E(\xi,\lambda))^{-1}\right)'
		\end{equation*}
		for certain multi-index $\alpha'$ with $|\alpha'| = n$. Since $E(\xi)$ and $E(\xi,\lambda)$ are second order polynomials in $\xi$, $E''(\xi)$ and $E''(\xi,\lambda)$ are constants which might be included into $a_{j,n}$. We apply one partial derivation on the right hand side of \eqref{eq:indukce}. For $j\neq \frac n2$ we get (throughout this proof we use the letter $a$ to denote a constant which may vary from line to line but it possesses the same dependencies as $a_{j,n}$)
		\begin{multline*}
		\left(\left(E'(\xi) + E'(\xi,\lambda)\right)^{n-2j} \left(\lambda \mathbb I_d + E(\xi) + E(\xi,\lambda)\right)^{-1-n+j}\right)' =\\
		a\left(E'(\xi) + E'(\xi,\lambda)\right)^{n-2j-1}\left(\lambda \mathbb I_d + E(\xi) + E(\xi,\lambda)\right)^{-1-n+j}\\
		+ a \left(E'(\xi) + E'(\xi,\lambda)\right)^{n-2j+1} \left(\lambda \mathbb I_d + E(\xi) + E(\xi,\lambda)\right)^{-2-n+j}.
		\end{multline*}
		Both terms can be found on the right hand side of relation \eqref{eq:indukce} assuming $|\alpha|= n+1$. 
		
		Now let $j=\frac n2$. Then
		\begin{multline*}
		\left((E'(\xi) + E'(\xi,\lambda))^0 \left(\lambda \mathbb I_d + E(\xi) + E(\xi,\lambda)\right)^{-1-\frac n2}\right)'\\
		= \left(E'(\xi) + E'(\xi,\lambda)\right) \left(\lambda \mathbb I_d + E(\xi) + E(\xi,\lambda)\right)^{-2-\frac n2}\\
		= \left(E'(\xi) + E'(\xi,\lambda)\right) \left(\lambda \mathbb I_d + E(\xi) + E(\xi,\lambda)\right)^{-1-(n+1) - \left\lfloor \frac{n+1}{2}\right\rfloor}
		\end{multline*}
		and this term is also on the right hand side of \eqref{eq:indukce} assuming $|\alpha|= n+1$.
	\end{proof}
	We use this lemma to infer
	\begin{equation*}
	|\partial^\alpha_\xi (\lambda \mathbb I_d   + E(\xi) + E(\xi,\lambda))^{-1}|\leq c^\alpha\left(|\lambda|^{1/2} + |\xi|\right)^{-2-|\alpha|}.
	\end{equation*}
	An easy calculation yield the validity of \eqref{eq:multipl}.
\end{proof}

\begin{proof}[Proof of Theorem \ref{thm.non-constant}] Recall that, according to the standard decomposition of unity, for all $\varepsilon>0$ there is a covering $\{B_j\}_{j=1}^\infty$ of the torus $\Omega$ such that
	\begin{itemize}
		\item $\bigcup_{j=1}^\infty B_j \supset \Omega$.
		\item There are functions $\zeta_j$ and $\tilde \zeta_j$ such that $\supp \zeta_j \subset B_j$, $\tilde \zeta_j = 1$ on $\supp \zeta_j$, $\supp \tilde \zeta_j\subset 2B_j$, $\|\zeta_j\|_{W^{2,\infty}} + \|\tilde\zeta_j\|_{W^{2,\infty}} \leq c$ with $c$ independent of $j$, and $\sum_{j=1}^\infty \zeta_j = 1$ on $\Omega$.
		\item Since $\nabla u_0, \theta_0\in BUC$, we may assume the following property: for every $j\in \mathbb N$ 
		\begin{equation*}
		|\mathbb Du_0(x) - \mathbb Du_0(x_j)| + |\theta_0(x) - \theta_0(x_j)| <\varepsilon
		\end{equation*}
		for all $x\in B_j$. Here $x_j$ is the center of a ball $B_j$. 
	\end{itemize}
	We emphasize that $\varepsilon>0$ will be chosen later. 
	
	For every $j$ we solve an equation
	\begin{multline}\label{eq:lokalni}
	\lambda u_j - \frac1{\varrho_* + \theta_0(x_j)}  \mathcal A(\mathbb D u_0(x_j))(\mathbb D u) -  \pi'(\varrho_* + \theta_0(x_j))\frac 1\lambda  \nabla \dvr u_j \\
	= \tilde \zeta_j\left(F + \left(\frac 1{\varrho_* + \theta_0(x)} \mathcal A(\mathbb D u(x))(\mathbb D u) - \frac 1{\varrho_* + \theta_0(x_j)}  \mathcal A(\mathbb D u_0(x_j))(\mathbb D u) \right)\right. \\
	+\left.\left(\pi'(\varrho_* + \theta_0(x)) \frac{1}\lambda \nabla\dvr u_j\right) - \pi'(\varrho_* + \theta_0(x_j))\frac 1\lambda  \nabla \dvr u_j\right).
	\end{multline}
	Since $\mathbb D u$ and $\theta$ are bounded and uniformly continuous function, \eqref{eq:lokalni} can be written as
	\begin{equation}\label{eq:lokal.operator}
	\lambda u_j - \mathscr B_\lambda u_j = \tilde \zeta_j F + \tilde\zeta_j \mathscr P_{\lambda j} u_j
	\end{equation}
	where $\mathscr P_{\lambda j}:W^{2,q}(\mathbb R^d)\to L^q(\mathbb R^d)$ is a linear operator bounded independently of $j$. Note also that $\mathscr B_\lambda +  \mathscr P_{\lambda j} = \mathcal B_\lambda$ on $B_j$. 
	
	We claim that $\lambda(\lambda \mathbb I_d - \mathscr B_\lambda -  \mathscr P_{\lambda j})^{-1}$, $|\lambda|^{\frac 12} \partial_k(\lambda \mathbb I_d - \mathscr B_\lambda -  \mathscr P_{\lambda j})^{-1}$ and $\partial_k\partial_l(\lambda \mathbb I_d - \mathscr B_\lambda -  \mathscr P_{\lambda j})^{-1}$ are $\mathcal R-$bounded on $\Sigma_{\beta,\nu}$ assuming $\varepsilon$ is small enough. To justify this claim it is enough to repeat the proof of \cite[Proposition 4.2]{DeHiPr}. First, recall that $\mathscr P_\lambda$ is $\mathcal R-$bounded on $\Sigma_{\beta,\nu}$ by $c\varepsilon$ according to Proposition \ref{pro.2.13.ensh}. The same proposition also yields $(\lambda \mathbb I_d - \mathscr B_\lambda)^{-1}$ is $\mathcal R-$bounded as $\lambda(\lambda\mathbb I_d-\mathscr B_\lambda)^{-1}$ is $\mathcal R-$bounded on $\Sigma_{\beta,\nu}$ by Theorem~\ref{Thm:ConstCase}. We have
	\begin{multline*}
	\lambda(\lambda \mathbb I_d -\mathscr B_\lambda -  \mathscr P_{\lambda j})^{-1} = \lambda (\lambda \mathbb I_d - \mathscr B_\lambda)^{-1} (\mathbb I_d -  \mathscr P_{\lambda j}(\lambda \mathbb I_d - \mathscr B_{\lambda})^{-1})^{-1}\\
	= \lambda (\lambda \mathbb I_d - \mathscr B_\lambda)^{-1} \sum_{n=0}^\infty (\mathscr P_{\lambda j} (\lambda\mathbb I_d - \mathscr B_\lambda)^{-1})^n.
	\end{multline*}
	We use Proposition \ref{pro.2.16.ensh} to deduce
	\begin{equation*}
	\mathcal R\left\{\lambda (\lambda \mathbb I_d - \mathscr B_\lambda)^{-1}  ( \mathscr P_{\lambda j} (\lambda\mathbb I_d - \mathscr B_\lambda)^{-1})^n\right\} \leq \mathcal R\left\{ \lambda (\lambda \mathbb I_d - \mathscr B_\lambda)^{-1}\right\}\left( c\varepsilon \mathcal R\left\{(\lambda \mathbb I_d - \mathscr B_\lambda)^{-1}\right\}\right)^{n}
	\end{equation*}
	Now it is enough to take $\varepsilon$ such that $\gamma:= c\varepsilon  \mathcal R\left\{(\lambda \mathbb I_d - \mathscr B_\lambda)^{-1}\right\}$ is less than $1$. This allows to claim that 
	\begin{equation*}
	\mathcal R\left\{\lambda(\lambda \mathbb I_d -\mathscr B_\lambda - \mathscr P_{\lambda j})^{-1}\right\}\leq \mathcal R\left\{\lambda(\lambda \mathbb I_d -\mathscr B_\lambda)^{-1}\right\}\frac 1{1-\gamma}.
	\end{equation*}
	The same true can be deduced also for the remaining two operators.

	Let $\mathcal T_j(\lambda)$ be a solution operator to \eqref{eq:lokal.operator}. We define
	\begin{equation*}
	\mathcal W(\lambda) (f) = \sum_{j=1}^\infty \zeta_j \mathcal T_j(\lambda)(\tilde \zeta_j f).
	\end{equation*}
	Then $u = \mathcal W(\lambda) (f)$ fulfills
	\begin{equation*}
	\lambda u - \mathcal B_\lambda u = f - \mathscr K_\lambda(f)
	\end{equation*}
	where $\mathscr K_\lambda:L^q(\Omega)\to L^q(\Omega)$ is defined as $\mathscr K_\lambda(f) = \sum_{j=1}^\infty\mathcal B_\lambda (\zeta_j u_j) - \zeta_j \mathcal B_\lambda (u_j)$ and $u_j = \mathcal T_j(\lambda)(\tilde\zeta_j f)$.
	The operator $\mathscr K_\lambda$ may be written as 
	\begin{multline}\label{eq:definice.k}
	\mathscr K_\lambda = \frac 1{\varrho_* + \theta_0}\sum_{j=1}^\infty\sum_{k,l,n=1}^d a_{mn}^{kl}(\mathbb D u_0) \left(\partial_k\partial_l \zeta_j u_{j,n} + \partial_l \zeta_j \partial_k u_{j,n} + \partial_k \zeta_j \partial_l u_{j,n}\right)\\
	+ \sum_{j=1}^\infty\pi'(\varrho_* + \theta_0) \frac 1\lambda \left(\nabla \zeta_j \dvr u_j + \nabla^2 \zeta_j u_j + \nabla \zeta_j \nabla u_j\right) 
	\end{multline}
	and $\mathscr K_\lambda$ is in particular $\mathcal R-$bounded on $\Sigma_{\beta,\nu'}$. Indeed, we show this particular boundedness of one term of the right hand side of \eqref{eq:definice.k} since the others might be treated in the same way.
	\begin{multline*}
	\int_0^1 \left\|\sum_{l=1}^n r_l(v) \left(\sum_{j=1}^\infty (\nabla \zeta_j)\dvr  \mathcal T_j(\lambda_l)(\tilde\zeta_j f)\frac{\pi'(\varrho_* + \theta_0)}{\lambda}\right) \right\|^q_{L^q(\Omega)}\ {\rm d}v\\
	\leq c \sum_{j=1}^\infty \int_0^1 \left\|\sum_{l=1}^n r_l(v)(\nabla \zeta_j) \dvr  \mathcal T_j(\lambda_l)(\tilde\zeta_j f)\right\|^q_{L^q(B_j\cap \Omega)}\ {\rm d}v\\
	\leq c \sum_{j=1}^\infty \int_0^1 \left\|\sum_{l=1}^n r_l(v)|\lambda_l|^{-1/2} |\lambda_l|^{1/2} \dvr  \mathcal T_j(\lambda_l)(\tilde\zeta_j f)\right\|^q_{L^q(B_j\cap \Omega)}\ {\rm d}v\\
	\leq c \nu'^{-1/2}\sum_{j=1}^\infty \int_0^1 \left\|\sum_{l=1}^n r_l(v) |\lambda_l|^{1/2} \dvr  \mathcal T_j(\lambda_l)(\tilde\zeta_j f)\right\|^q_{L^q(B_j\cap \Omega)}\ {\rm d}v\\
	\leq c \nu'^{-1/2}\sum_{j=1}^\infty \int_0^1 \left\|\sum_{l=1}^n r_l(v)  f\right\|^q_{L^q(B_j\cap \Omega)}\ {\rm d}v.
	\end{multline*}
	Consequently, the $\mathcal R-$bound of $\mathscr K_\lambda$ is $c\nu'^{-1/2}$ and by a proper choice of $\nu'$ it is less or equal than $\frac 12$. We deduce that 
	\begin{equation*}
	\mathcal R_{\mathcal L(L^q)}\{(\mathbb I_d-\mathscr K_\lambda)^{-1}, \lambda \in \Sigma_{\beta,\nu'}\}\leq 2.
	\end{equation*}
	See also Proposition \ref{pro.2.16.ensh}.
	We thus have $u = \mathcal W\circ (\mathbb I_d-\mathscr K_\lambda)^{-1} (f)$ and this solution operator fulfills all the demanded properties. 
	
	Note also that the solution to \eqref{eq:hvezdicka} is unique. This comes from the ellipticity of operator $-\mathcal B_\lambda$.
\end{proof}

\begin{proof}[Proof of Theorem \ref{thm:regularita}]
	Due to Theorem \ref{thm.non-constant} we have
	\begin{equation}\label{LResRBound}
	\lambda(\lambda \mathbb I_d - \mathscr A)^{-1}\ \mathcal R-\text{bounded for all }\lambda \in \Sigma_{\beta,\nu}.
	\end{equation}
	Especially, $\lambda(\lambda \mathbb I_d -\mathscr A)^{-1}$ is bounded and by the Hille-Yosida theorem (see \cite[Theorem 3.1, Chapter 1]{viorel}) there exists a  semigroup of class $C^0$ generated by $\mathscr A$. We denote this semigroup $\mathscr T$, i.e.
	\begin{equation*}
	\pat \mathscr T -\mathscr A\mathscr T = 0. 
	\end{equation*}
	We set $ U(t) = \int_0^t \mathscr T(t-s) F(s)\ {\rm d}s$ where $F = (\mathscr G, \mathscr F)$ and we immediately get
	\begin{equation}\label{eq:trik.s.nu}
	\|U\|_{L^p(0,T;(W^{1,q}(\Omega)\times L^q(\Omega)))} \leq c \|F\|_{L^p(0,T;W^{1,q}(\Omega)\times L^q(\Omega))}
	\end{equation}
	where $c$ depends also on the time interval $(0,T)$. Recall that $U$ has two components and we denote them by $\theta$ and $u$. Moreover $U$ solves
	\begin{equation*}
	\pat U - \mathscr A U = F
	\end{equation*}
	which is another form of \eqref{eq:A.operator}.
	
	Clearly, $U$ solves also
	\begin{equation*}
	\pat U + 2\nu U  -  \mathscr A U = F + 2\nu U.
	\end{equation*}
	We set $G = F + 2\nu U$ and we extend the definition of $G$ and $U$ in such a way that $U(t) = 0$, $G(t) = 0$ for all $t<0$. 
	
	We may write
	\begin{equation*}
	U(t) = \int_{-\infty}^\infty e^{(-2\nu\mathbb I_d + \mathscr A)(t-s)}\chi_{[0,\infty)}(t-s) G(s) \ {\rm d}s = \left(e^{(-2\nu\mathbb I_d +\mathscr A)(s)}  \chi_{[0,\infty)}(s)\right)* \left(G(s)\right) (t).
	\end{equation*}
	
	By rules for Fourier transform we get
	\begin{equation*}
	\widehat{\pat U} = i\xi \left((-i\xi-2\nu)\mathbb I_d + \mathscr A\right)^{-1} \widehat G 
	\end{equation*}
	and $M(\xi)=i\xi\left((-i\xi-2\nu)\mathbb I_d+\mathscr A\right)^{-1}$ is the corresponding multiplier as mentioned in Theorem \ref{thm.weis}. Indeed, using \eqref{LResRBound}, Proposition \ref{pro.2.13.ensh} and Proposition \ref{pro.2.16.ensh} we infer that $M(\xi)$ and $\xi M'(\xi)$ are $\mathcal{R}$--bounded and Theorem \ref{thm.weis} implies 
	\begin{equation*}
	\|\pat U\|_{L^p(0,T;W^{1,q}(\Omega)\times L^q(\Omega))} \leq c\|G\|_{L^p(0,T;W^{1,q}(\Omega)\times L^q(\Omega))} \leq c \|F\|_{L^p (0,T;(W^{1,q}(\Omega)\times L^q(\Omega)))}
	\end{equation*}
	having \eqref{eq:trik.s.nu} in mind. Consequently,
	\begin{equation}\label{eq:prvni.vysledek}
	\|\theta\|_{W^{1,p}(0,T;W^{1,q}(\Omega))} + \| u\|_{W^{1,p}(0,T;L^q(\Omega))}  \leq c\left(\|\mathscr G\|_{L^p(0,T;W^{1,q}(\Omega))} + \|\mathscr F\|_{L^p(0,T;L^q(\Omega))}\right).
	\end{equation}
	
	Note also that $\theta,u$ solves 
	\begin{equation*}
	2\nu\left(\begin{matrix}\theta\\ u\end{matrix}\right) - \mathscr A \left(\begin{matrix} \theta\\ u \end{matrix}\right) = \left(\begin{matrix} \mathscr G \\ \mathscr F\end{matrix}\right) - \pat \left(\begin{matrix} \theta \\ u \end{matrix}\right) + 2\nu\left(\begin{matrix}\theta\\ u \end{matrix} \right).
	\end{equation*}
	
	Theorem \ref{thm.non-constant} yields that $\partial_j\partial_k \lambda^{-1}\mathcal R_{\lambda 2}$ is bounded on $\Sigma_{\beta,\nu}$ for every $j,k= 1,\ldots,d$ and thus we deduce (with help of \eqref{eq:prvni.vysledek})
	\begin{equation*}
	\|u\|_{L^p(0,T;W^{2,q}(\Omega))}\leq c\left(\|\mathscr G\|_{L^p(0,T;W^{1,q}(\Omega))} + \|\mathscr F\|_{L^p(0,T;L^{q}(\Omega))}\right).
	\end{equation*}
	
	This completes the proof.
	
\end{proof}

\section{Appendix}

Here we present several theorems from other sources for readers convenience. 

\begin{Theorem}[Theorem 3.3 in \cite{EnSh}]\label{thm.3.3.ensh}
	  Let $1<q<\infty$ and let $\Lambda\subset \mathbb C$. Let $m(\lambda,\xi)$ be a function defined on $\Lambda \times (\mathbb R^d\setminus \{0\})$ such that for any multi-index $\alpha \in \mathbb N^d_0$ there exists a constant $C_\alpha$ depending on $\alpha$ and $\Lambda$ such that 
	$$ |\partial^\alpha_\xi m(\lambda,\xi)|\leq C_\alpha |\xi|^{-\alpha}
	$$
	for any $(\lambda,\xi)\in \Lambda\times (\mathbb R^d\setminus \{0\})$. Let $K_\lambda$ be an operator defined by $K_\lambda f = \mathcal F^{-1}_\xi [m(\lambda,\xi)\hat f(\xi)]$. Then, the set $\{K_\lambda|\lambda\in \Lambda\}$ is $\mathcal R-$bounded on $\mathcal L(L^q(\mathbb R^d))$ and 
	$$
	\mathcal R_{\mathcal L(L^q(\mathbb R^d))} (\{K_\lambda | \lambda\in\Lambda\}) \leq C_{q,d} \max_{|\alpha|\leq d+2} C_\alpha.
	$$
\end{Theorem}

\begin{Lemma}
	[Lemma 3.1 in \cite{EnSh}] \label{lem.3.1.ensh} Let $0<\beta<\pi/2$ and $\nu_0>0$. For any $\lambda\in \Sigma_{\beta,\nu}$ we have
	$$
	|\lambda + |\xi|^2| \geq \sin(\beta/2) (|\lambda| + |\xi|^2).
	$$
\end{Lemma}

\begin{Proposition}
	[Proposition 2.13 in \cite{EnSh}] \label{pro.2.13.ensh} Let $D\subset \mathbb R^d$ be a domain and let $\Lambda$ be a domain in $\mathbb C$. Let $m(\lambda)$ be a bounded function on $\Lambda$ and let $M_m(\lambda):L^q(D)\mapsto L^q(D)$ be defined as $M_m(\lambda) f = m(\lambda)f$. Then
	$$
	\mathcal R_{\mathcal L(L^q(D))} (\{M_m(\lambda)|\lambda\in \Lambda\}) \leq C_{d,q,D} \|m\|_{L^\infty(\Lambda)}.
	$$
\end{Proposition}

\begin{Proposition}
	[Proposition 2.16 in \cite{EnSh} or Proposition 3.4 in \cite{DeHiPr}] \label{pro.2.16.ensh}
	\hspace{0.5em} \begin{enumerate}\item Let $X$ and $Y$ be Banach spaces and let $\mathcal T$ and $\mathcal S$ be $\mathcal R-$bounded families on $\mathcal L(X,Y)$. Then $\mathcal T+\mathcal S = \{T+S| T\in \mathcal T, S\in \mathcal S\}$ is also $\mathcal R-$bounded on $\mathcal L(X,Y)$ and
		$$
		\mathcal R_{\mathcal L(X,Y)} (\mathcal T + \mathcal S)\leq \mathcal R_{\mathcal L(X,Y)} (\mathcal T) + R_{\mathcal L(X,Y)} (\mathcal S).
		$$ 
		\item Let $X,Y$ and $Z$ be Banach spaces and let $\mathcal T$ and $\mathcal S$ be $\mathcal R-$bounded families on $\mathcal L(X,Y)$ and $\mathcal L(Y,Z)$ respectively. Then $\mathcal S\mathcal T = \{ST| T\in \mathcal T, S\in \mathcal S\}$ is $\mathcal R-$bounded on $\mathcal L(X,Y)$ and
		$$
		\mathcal R_{\mathcal L(X,Z)}(\mathcal S\mathcal T) \leq \mathcal R_{\mathcal L(X,Y)} (\mathcal T) \mathcal R_{\mathcal L(Y,Z)}(\mathcal S).
		$$
	\end{enumerate}
	
\end{Proposition}
\begin{Remark}
Our plan is to extend the result to the case of a bounded domain including some more general boundary conditions as inflow/outflow or nonhomogeneous Dirichlet boundary conditions.

\end{Remark}
{\bf Acknowledgement}: The research of \v{S}.N. and V.M. leading to these results has received funding from the Czech Sciences Foundation (GA\v CR), GA19-04243S and in the framework of RVO: 67985840. The research of M.K. was supported by RVO: 67985840. Moreover,  M.K. was supported by the Czech Sciences Foundation (GA\v CR), GA19-04243S during his research on the final version.
\bigskip

{\bf Conflict of interest:}

\v{S}\'arka Ne\v{c}asov\'a as 
 the corresponding author declares on behalf of all authors, that there is no conflict of interest.

{\bf Declarations:}
'Not applicable' for the whole manuscript.

{\bf Data availability statement:}
There are no associated data corresponding to the manuscript.
\bibliographystyle{plain}
\bibliography{literatura}

\end{document}